\documentclass[a4paper,reqno]{amsart}
\usepackage{amssymb}
\usepackage{pstricks}
\begin{document}
\title[Finding generators and relations]{Finding generators and relations for \\ groups acting on the hyperbolic ball}

\author[Donald I. Cartwright \\ Tim Steger]{Donald I. Cartwright \\ Tim Steger}

\addtolength{\baselineskip}{-0.5pt}
\numberwithin{equation}{section}
\newtheorem{theorem}{Theorem}[section]
\newtheorem{proposition}{Proposition}[section]
\newtheorem{lemma}{Lemma}[section]
\newtheorem{corollary}{Corollary}[section]
\theoremstyle{definition}
\newtheorem{definition}{Definition}
\theoremstyle{remark}
\newtheorem{remark}{Remark}
\newtheorem{example}{Example}
\newcommand{\C}{{\mathbb C}}
\newcommand{\F}{{\mathbb F}}
\newcommand{\R}{{\mathbb R}}
\newcommand{\Q}{{\mathbb Q}}
\newcommand{\Z}{{\mathbb Z}}
\newcommand{\cC}{\mathcal{C}}
\newcommand{\cF}{\mathcal{F}}
\newcommand{\cG}{\mathcal{G}}
\newcommand{\vol}{{\mathrm{vol}}}
\begin{abstract}In order to enumerate the fake projective planes, as announced in~\cite{CS},
we found explicit generators and a presentation for each maximal arithmetic subgroup
$\bar\Gamma$ of~$PU(2,1)$ for which the (appropriately normalized) covolume equals~$1/N$ 
for some integer~$N\ge1$. Prasad and Yeung \cite{PY1,PY2} had given a list of all 
such $\bar\Gamma$ (up to equivalence).

The generators were found by a computer search which uses the natural action of $PU(2,1)$ on
the unit ball $B(\C^2)$ in~$\C^2$. Our main results here give criteria
which ensure that the computer search has found sufficiently many elements 
of~$\bar\Gamma$ to generate $\bar\Gamma$, and describes a family of relations amongst
the generating set sufficient to give a presentation of~$\bar\Gamma$.

We give an example illustrating details of how this was done in the case of a particular~$\bar\Gamma$
(for which $N=864$). While there are no fake projective planes in this case, we exhibit
a torsion-free subgroup~$\Pi$ of index~$N$ in~$\bar\Gamma$, and give some properties of the surface~$\Pi\backslash B(\C^2)$.
\end{abstract}
\maketitle
\begin{section}{Introduction}
Suppose that $P$ is a fake projective plane. Its Euler-characteristic~$\chi(P)$ is~3. The fundamental 
group $\Pi=\pi_1(P)$ embeds as a cocompact arithmetic lattice subgroup of~$PU(2,1)$, and so 
acts on the unit ball $B(\C^2)=\{(z_1,z_2)\in\C^2:|z_1|^2+|z_2|^2<1\}$ in~$\C^2$, endowed with the
hyperbolic metric. Let $\cF$ be a fundamental domain for this
action. There is a normalization of the hyperbolic volume~$\vol$ on~$B(\C^2)$ and of the Haar measure~$\mu$
on~$PU(2,1)$ so that $\chi(P)=3\vol(\cF)=3\mu(PU(2,1)/\Pi)$. So $\mu(PU(2,1)/\Pi)=1$. Let 
$\bar\Gamma\le PU(2,1)$ be maximal arithmetic, with $\Pi\le\bar\Gamma$. Then $\mu(PU(2,1)/\bar\Gamma)=1/N$ 
and $[\bar\Gamma:\Pi]=N$ for some integer~$N\ge1$. 

The fundamental group $\Pi$ of a fake projective plane must also be torsion-free and $\Pi/[\Pi,\Pi]$ 
must be finite (see \cite{CS} for example). 

As announced in~\cite{CS}, we have found all groups~$\Pi$ with these properties, up to isomorphism.
Our method was to find explicit generators and an explicit presentation for each~$\bar\Gamma$,
so that the question of finding all subgroups~$\Pi$ of~$\bar\Gamma$ with index~$N$ and
the additional required properties, as just mentioned, can be studied.

In Section~2, we give results about finding generators and relations for groups~$\Gamma$
acting on quite general metric spaces~$X$. The main theorem gives simple conditions which ensure that a 
set~$S$ of elements of~$\Gamma$, all of which move a base point~$0$ by at most a certain distance~$r_1$,
are all the elements of~$\Gamma$ with this property. 

In Section~3 we specialize to the case $X=B(\C^2)$, and treat in detail a particular group~$\Gamma$.
This~$\Gamma$ is one of the maximal arithmetic subgroups $\bar\Gamma\le PU(2,1)$ listed in~\cite{PY1,PY2} 
with covolume of the form~$1/N$, $N$ an integer. In this case~$N=864$. Consider the action of~$\Gamma$
on~$B(\C^2)$, and let 0 denote the origin in~$B(\C^2)$. Two elements,
denoted $u$ and~$v$, generate the stabilizer~$K$ of~0 in~$\Gamma$. Another element~$b$ of~$\Gamma$ was found by
a computer search looking for elements $g\in\Gamma$ for which $d(g.0,0)$ is small.
We use the results of Section~2 to show that the computer search did not miss any such~$g$,
and to get a simple presentation of~$\Gamma$. It turns out that this~$\Gamma$ is one of the
Deligne-Mostow groups (see Parker~\cite{Parker}).

We apply this presentation of~$\Gamma$ to exhibit a torsion-free subgroup~$\Pi$ of index~864. The
abelianization $\Pi/[\Pi,\Pi]$ is~$\Z^2$, and so $\Pi$ is not the fundamental group of a fake-projective
plane. However the ball quotient $\Pi\backslash B(\C^2)$ is a compact complex surface with interesting
properties, some of which we describe. By a lengthy computer search not discussed here, we showed
that any torsion-free subgroup of index~864 in~$\Gamma$ is conjugate to~$\Pi$, and so no fake projective
planes arise in this context.

In Section~4, we calculate the value of~$r_0$ for the example of the previous section.
\end{section}
\begin{section}{General results}\label{sec:generalresults}
Let $\Gamma$ be a group acting by isometries on a simply-connected geodesic metric space~$X$.
Let $S\subset\Gamma$ be a finite symmetric generating set for~$\Gamma$.
Fix a point $0\in X$, and suppose that $d(0,x)$ is bounded on the set 
\begin{equation}\label{eq:fsdefinition}
\cF_S=\{x\in X:d(0,x)\le d(g.0,x)\ \text{for all}\ g\in S\}.
\end{equation}
Define
\begin{displaymath}
r_0=\sup\{d(0,x):x\in\cF_S\}.
\end{displaymath}
\begin{theorem}\label{thm:mainresult}Suppose that there is a number $r_1>2r_0$ such that
\begin{itemize}
\item[(a)] if $g\in S$, then $d(g.0,0)\le r_1$,
\item[(b)] if $g,g'\in S$ and $d((gg').0,0)\le r_1$, then $gg'\in S$.
\end{itemize}
Then $S=\{g\in\Gamma:d(g.0,0)\le r_1\}$.
\end{theorem}
\begin{corollary}\label{cor:funddomain}For $\Gamma,S$ as in Theorem~\ref{thm:mainresult}, $\cF_S$ is equal to the Dirichlet funda\-ment\-al domain
$\cF=\{x\in X:d(0,x)\le d(g.0,x)\ \text{for all}\ g\in\Gamma\}$
of~$\Gamma$ centered at~$0$.
\end{corollary}
\begin{proof}Clearly $\cF\subset\cF_S$. If $x\in\cF_S\setminus\cF$,
pick a $g\in\Gamma$ so that $d(g.0,x)<d(0,x)$. Now $d(0,x)\le r_0$
because $x\in\cF_S$, and so $d(0,g.0)\le d(0,x)+d(x,g.0)\le2d(0,x)\le2r_0\le r_1$, 
so that $g\in S$, by Theorem~\ref{thm:mainresult}. But then $d(0,x)\le d(g.0,x)$, a contradiction.
\end{proof}
\begin{corollary}Assume $\Gamma,S$ are as in Theorem~\ref{thm:mainresult},
and that the action of each\break
$g\in\Gamma\setminus\{1\}$ is nontrivial,
so that $\Gamma$ may be regarded as a subgroup of the group~$\cC(X)$ of 
continuous maps $X\to X$. Then $\Gamma$ is discrete for the compact open topology.
\end{corollary}
\begin{proof}Let $V_1=\{f\in\cC(X):f(0)\in B_{r_1}(0)\}$, where 
$B_r(x)=\{y\in X:d(x,y)<r\}$. Theorem~\ref{thm:mainresult} shows that $\Gamma\cap V_1\subset S$.
For each $g\in S\setminus\{1\}$, choose $x_g\in X$ so that $g.x_g\ne x_g$, let 
$r_g=d(g.x_g,x_g)$, and let $V_g=\{f\in\cC(X):f(x_g)\in B_{r_g}(x_g)\}$.
Then $g\not\in V_g$. Hence the intersection~$V$ of the sets $V_g$, $g\in S$,
is an open neighborhood of~$1$ in~$\cC(X)$ such that $\Gamma\cap V=\{1\}$.
\end{proof}
Before starting the proof of Theorem~\ref{thm:mainresult}, we need some
lemmas, which have the same hypotheses as Theorem~\ref{thm:mainresult}.

\begin{lemma}\label{lem:tworzero}
The group $\Gamma$ is generated by $S_0=\{g\in S : d(g.0,0)\le2r_0\}$.
\end{lemma}
\begin{proof}As $S$ is finite, there is a $\delta>0$ 
so that $d(g.0,0)\ge 2r_0+\delta$ for all $g\in S\setminus S_0$. Let $\Gamma_0$ denote 
the subgroup of~$\Gamma$ generated by~$S_0$, and assume that $\Gamma_0\subsetneqq\Gamma$. 
Since $\Gamma$ is generated by~$S$, there are $g\in S\setminus\Gamma_0$, and we
choose such a $g$ with $d(g.0,0)$ as small as possible.
Then $d(g.0,0)>2r_0$, since otherwise $g\in S_0\subset\Gamma_0$.
In particular, $g.0\not\in\cF_S$. Since $0\in\cF_S$, there is a last point~$\xi$ belonging to~$\cF_S$
on the geodesic from~$0$ to~$g.0$. Choose any point $\xi'$ on that geodesic
which is outside~$\cF_S$ but satisfies $d(\xi,\xi')<\delta/2$.
As $\xi'\not\in\cF_S$, there is an $h\in S$ such that $d(h.0,\xi')<d(0,\xi')$. 
Hence
\begin{displaymath}
d(h.0,g.0)\le d(h.0,\xi')+d(\xi',g.0)<d(0,\xi')+d(\xi',g.0)=d(0,g.0),
\end{displaymath}
so that $d((h^{-1}g).0,0)<d(g.0,0)$. So $h^{-1}g$ cannot be in $S\setminus\Gamma_0$, by choice of~$g$.
Also,
\begin{displaymath}
d(h.0,0)\le d(h.0,\xi')+d(\xi',0)<2d(0,\xi')\le2(d(0,\xi)+d(\xi,\xi'))<2r_0+\delta.
\end{displaymath}
Hence $h\in S_0$, by definition of~$\delta$. Now $h^{-1}g\in S$ by hypothesis~(b) above, 
since $h^{-1},g\in S$ and $d((h^{-1}g).0,0)<d(0,g.0)\le r_1$. So $h^{-1}g$ must be in~$\Gamma_0$.
But then $g=h(h^{-1}g)\in\Gamma_0$, contradicting our assumption.
\end{proof}

\begin{lemma}\label{lem:remepsilon}If $x\in X$ and if $d(0,x)\le r_0+\epsilon$,
where $0<\epsilon\le(r_1-2r_0)/2$, then there exists $g\in S$ such that $g.x\in\cF_S$,
and in particular, $d(0,g.x)\le r_0$.
\end{lemma}
\begin{proof}Since $S$ is finite, we can choose $g\in S$ so that $d(0,g.x)$ is minimal. If $g.x\in\cF_S$, there is nothing to prove, and
so assume that $g.x\not\in\cF_S$. There exists $h\in S$ such that $d(h.0,g.x)<d(0,g.x)$, and so
$d(0,(h^{-1}g).x)=d(h.0,g.x)<d(0,g.x)$. By the choice of~$g$, $h^{-1}g$ cannot be in~$S$, and also $d(0,g.x)\le d(0,x)$. So
\begin{displaymath}
\begin{aligned}
d(0,(h^{-1}g).0)\le d(0,(h^{-1}g).x)+d((h^{-1}g).x,(h^{-1}g).0)&=d(0,(h^{-1}g).x)+d(x,0)\\
&<d(0,g.x)+d(x,0)\\
&\le2d(0,x)\\
&\le 2r_0+2\epsilon\le r_1.
\end{aligned}
\end{displaymath}
But this implies that $h^{-1}g\in S$ by~(b) again, since $g,h\in S$. This 
is a contradiction, and so $g.x$ must be in $\cF_S$.
\end{proof}
\begin{remark}If in Lemma~\ref{lem:remepsilon} we also assume that $\epsilon<\delta/2$, where $\delta$ is as in
the proof of Lemma~\ref{lem:tworzero}, then the $g$ in the last lemma can be chosen
in~$S_0$. For then $d(0,g.0)\le d(0,g.x)+d(g.x,g.0)\le2d(0,x)\le2(r_0+\epsilon)<2r_0+\delta$.
\end{remark}

To prove Theorem~\ref{thm:mainresult}, we consider a $g\in\Gamma$ such that $d(g.0,0)\le r_1$. By Lemma~\ref{lem:tworzero},
we can write $g=y_1y_n\cdots y_n$, where $y_i\in S_0$ for each~$i$. Since $1\in S_0$, we
may suppose that $n$ is even, and write $n=2m$. Moreover, we can assume that $m$ is odd. 
Form an equilateral triangle~$\Delta$ in the Euclidean plane with horizontal base whose 
sides are each divided into~$m$ equal segments, marked off by 
vertices $v_0,v_1,\ldots,v_{3m-1}$. We use these to partition~$\Delta$ into $m^2$ subtriangles.
We illustrate in the case $m=3$: 
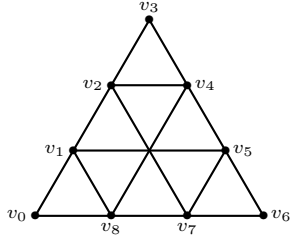
\begin{center}
\begin{pspicture}(0,-0.5)(3,3)
\psset{xunit=1,yunit=1}%
\psline[linecolor=black]{-}(0,0)(1.5,2.598)(3,0)(0,0)
\psline[linecolor=black]{-}(0.5,0.866)(2.5,0.866)
\psline[linecolor=black]{-}(1.0,1.732)(2.0,1.732)
\psline[linecolor=black]{-}(0.5,0.866)(1.0,0)(2.0,1.732)
\psline[linecolor=black]{-}(1.0,1.732)(2.0,0)(2.5,0.866)
\rput[c](0,0){$\scriptstyle\bullet$}
\rput[c](0.5,0.866){$\scriptstyle\bullet$}
\rput[c](1.0,1.732){$\scriptstyle\bullet$}
\rput[c](1.5,2.598){$\scriptstyle\bullet$}
\rput[c](2.0,1.732){$\scriptstyle\bullet$}
\rput[c](2.5,0.866){$\scriptstyle\bullet$}
\rput[c](3,0){$\scriptstyle\bullet$}
\rput[c](2,0){$\scriptstyle\bullet$}
\rput[c](1,0){$\scriptstyle\bullet$}
\rput[r](-0.1,0){$\scriptstyle{v_0}$}
\rput[r](0.4,0.866){$\scriptstyle{v_1}$}
\rput[r](0.9,1.732){$\scriptstyle{v_2}$}
\rput[b](1.5,2.698){$\scriptstyle{v_3}$}
\rput[l](2.1,1.732){$\scriptstyle{v_4}$}
\rput[l](2.6,0.866){$\scriptstyle{v_5}$}
\rput[l](3.1,0){$\scriptstyle{v_6}$}
\rput[t](2,-0.1){$\scriptstyle{v_7}$}
\rput[t](1,-0.1){$\scriptstyle{v_8}$}
\end{pspicture}
\end{center}
We define a continuous function $\varphi$ from the boundary of~$\Delta$ 
to~$X$ which maps $v_0$ to~$0$ and $v_i$ to $(y_1\cdots y_i).0$ for $i=1,\ldots,n$,
which for $i=1,\ldots,n$ maps the segment $[v_{i-1},v_i]$ to the geodesic from~$\varphi(v_{i-1})$
to~$\varphi(v_i)$, and which maps the bottom side of~$\Delta$ to the geodesic from 
$g.0=(y_1\cdots y_n).0$ to~0.

Because $X$ is simply-connected, we can extend $\varphi$ to a continuous map 
$\varphi:\Delta\to X$, where $\Delta$ here refers to the triangle and its interior.

Let $\epsilon>0$ be as in Lemma~\ref{lem:remepsilon}. For a sufficiently large integer~$r$, which
we choose to be odd, by partitioning each of the above subtriangles of~$\Delta$ into $r^2$ 
congruent subtriangles, we have $d(\varphi(t),\varphi(t'))<\epsilon$ for all 
$t,t'$ in the same (smaller) subtriangle.

We now wish to choose elements $x(v)\in\Gamma$ for each vertex~$v$ of each of the subtriangles,
so that 
\begin{itemize}
\item[(i)] $d(x(v).0,\varphi(v))\le r_0$ for each~$v$ not in the interior of the bottom side of~$\Delta$,
\item[(ii)] $d(x(v).0,\varphi(v))\le r_1/2$ for the $v\ne v_0,v_n$ on the bottom side of~$\Delta$.
\end{itemize}
We shall also define elements $y(e)$ for 
each directed edge~$e$ of each of the subtriangles in such a way that 
\begin{itemize}
\item[(iii)] $x(w)=x(v)y(e)$ if $e$ is the edge from~$v$ to~$w$.
\end{itemize} 
Note that the $y(e)$'s are completely determined by the $x(v)$'s, and the $x(v)$'s are
completely determined from the $y(e)$'s provided that one $x(v)$ is specified.

We first choose $x(v)$ for the original vertices~$v=v_i$ on the left
and right sides of~$\Delta$ by setting $x(v_0)=1$ and
$x(v_i)=y_1\cdots y_i$ for $i=1,\ldots,n$. Thus
$d(x(v).0,\varphi(v))=0\le r_0$ for these $v$'s. Now if
$w_0=v_{i-1},w_1,\ldots,w_r=v_i$ are the $r+1$ equally spaced vertices
of the edge from~$v_{i-1}$ to~$v_i$, we set $x(w_j)=x(v_{i-1})$ for
$1\le j<r/2$ and $x(w_j)=x(v_i)$ for $r/2<j\le r-1$. Then if $1\le
j<r/2$, we have

\begin{displaymath}
\begin{aligned}
d(x(w_j).0,\varphi(w_j))=d(x(v_{i-1}).0,\varphi(w_j))&=d(\varphi(v_{i-1}),\varphi(w_j))\\
&\le\frac{1}{2}d(\varphi(v_{i-1}),\varphi(v_i))\quad(*)\\
&=\frac{1}{2}d(0,y_i.0)\le r_0,\\
\end{aligned}
\end{displaymath}
where the inequality~($*$) holds as $\varphi$ on the segment from~$v_{i-1}$ to~$v_i$
is the geodesic from $\varphi(v_{i-1})$ to~$\varphi(v_i)$. In the same way,
$d(x(w_j).0,\varphi(w_j))\le r_0$ for $r/2<j\le r-1$.

Having chosen $x(v)$ for the $v$ on the left and right sides of~$\Delta$, we set $y(e)=x(v)^{-1}x(w)$
if $e$ is the edge from~$v$ to~$w$ on one of those sides. So of the $r$ edges in the segment
from~$v_{i-1}$ to~$v_i$, we have $y(e)=1$ except for the middle edge, for which $y(e)=y_i$.

For the vertices $v$ on the bottom side of~$\Delta$, we set $x(v)=1$ if $v$ is closer to~$v_0$ than
to~$v_n$, and set $x(v)=g=y_1\cdots y_n$ otherwise, and we set $y(e)=x(v)^{-1}x(w)$
if $e$ is the edge from~$v$ to~$w$. Since $rm$ is odd,
there is no middle vertex on the side, so there is no ambiguity in the definition of~$x(v)$,
but there is a middle edge~$e$, and $y(e)=g$ if $e$ is directed from left to right. For the
other edges $e'$ on the bottom side, $y(e')=1$. If $v$ is a vertex of the bottom side of~$\Delta$
closer to~$v_0$ than to~$v_n$, then
\begin{displaymath}
d(x(v).0,\varphi(v))=d(0,\varphi(v))\le\frac{1}{2}d(0,g.0)\le\frac{r_1}{2},
\end{displaymath}
since $\varphi$ on the bottom side is just the geodesic from~$0$ to~$g.0$. Similarly,
if $v$ is a vertex of the bottom side of~$\Delta$ closer to~$v_n$, then again
$d(x(v).0,\varphi(v))\le r_1/2$.

We now choose $x(v)$ for vertices which are not on the sides of~$\Delta$ as follows.
We successively define $x(v)$ as we move from left to right on a horizontal line.
Suppose that $v$ is a vertex for which $x(v)$ has not yet been chosen, but for which
$x(v')$ has been chosen for the vertex~$v'$ immediately to the left of~$v$. Now
$d(\varphi(v),\varphi(v'))\le\epsilon$ and $d(x(v').0,\varphi(v'))\le r_0$, so that
$d(\varphi(v),x(v').0)\le r_0+\epsilon$. By Lemma~\ref{lem:remepsilon}, applied to
$x=x(v')^{-1}\varphi(v)$, there is a $g\in S$ so that $d(gx(v')^{-1}\varphi(v),0)\le r_0$.
So we set $x(v)=x(v')g^{-1}$. If $e$ is the edge from $v'$ to~$v$, we set $y(e)=g^{-1}\in S$.

We have now defined $x(v)$ for each vertex of the partitioned
triangle~$\Delta$ so that (i) and~(ii) hold, and these determine $y(e)$
for each directed edge~$e$ so that (iii) holds. We have seen that $y(e)\in
S$ if $e$ is an edge on the left or right side of~$\Delta$ or if $e$
is a horizontal edge not lying on the bottom side of~$\Delta$ whose
right hand endpoint does not lie on the right side of~$\Delta$. Also,
$y(e)=1\in S$ for all edges~$e$ in the bottom side of~$\Delta$ except the
middle one, for which $y(e)=g$. The theorem will be proved once we
check that $y(e)\in S$ for each edge~$e$.

\begin{lemma}\label{lem:thirdedge}Suppose that $v$, $v'$ and~$v''$ are the vertices of a subtriangle in~$\Delta$,
and that $e$, $e'$ and~$e''$ are the edges from $v$ to~$v'$, $v'$ to~$v''$, and $v$ to~$v''$,
respectively. Suppose that $y(e)$ and $y(e')$ are in~$S$. Then $y(e'')\in S$ too.
\end{lemma}
\begin{proof}We have $y(e)=x(v)^{-1}x(v')$, $y(e')=x(v')^{-1}x(v'')$ and $y(e'')=x(v)^{-1}x(v'')$,
and so $y(e'')=y(e)y(e')$. If $e''$ is the middle edge of the bottom side of~$\Delta$, 
then $d(y(e'').0,0)=d(g.0,0)\le r_1$ by hypothesis, and so $y(e'')\in S$ by~(b) above.
If $e''$ is any other edge of the bottom side of~$\Delta$, then $y(e'')=1\in S$. So
we may assume that $e''$ is not on the bottom side of~$\Delta$. Hence at most one of $v$ and~$v''$
lies on that bottom side. Hence
\begin{displaymath}
\begin{aligned}
d(0,y(e'').0)&=d(x(v).0,x(v)y(e'').0)\\
&=d(x(v).0,x(v'').0)\\
&\le d(x(v).0,\varphi(v))+d(\varphi(v),\varphi(v''))+d(\varphi(v''),x(v'').0)\\
&<d(x(v).0,\varphi(v))+d(\varphi(v''),x(v'').0)+\epsilon\\
&\le r_0+r_1/2+\epsilon\quad(\dagger)\\
&\le r_1,
\end{aligned}
\end{displaymath}
where the inequality $(\dagger)$ holds because at least one of $d(x(v).0,\varphi(v))\le r_0$
and $d(\varphi(v''),x(v'').0)\le r_0$ holds, since at most one of $v$ and~$v''$
is a vertex of the bottom side of~$\Delta$, and if say $v$ is such a 
point, then we still have $d(x(v).0,\varphi(v))\le r_1/2$.
\end{proof}
\begin{proof}[Conclusion of the proof of Theorem~\ref{thm:mainresult}]
We must show that $y(e)\in S$ for all edges. We use Lemma~\ref{lem:thirdedge}, working
down from the top of~$\Delta$, and moving from left to right. So in the order indicated
in the next diagram we work down $\Delta$, finding that $y(e)\in S$ in each case, until we get to the
lowest row of triangles. 
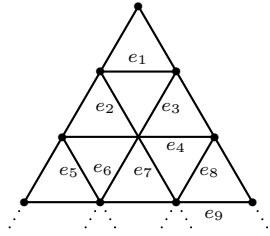
\begin{center}
\begin{pspicture}(0,-0.5)(3,3)
\psset{xunit=1,yunit=1}%
\psline[linecolor=black]{-}(0,0)(1.5,2.598)(3,0)(0,0)
\psline[linecolor=black]{-}(0.5,0.866)(2.5,0.866)
\psline[linecolor=black]{-}(1.0,1.732)(2.0,1.732)
\psline[linecolor=black]{-}(0.5,0.866)(1.0,0)(2.0,1.732)
\psline[linecolor=black]{-}(1.0,1.732)(2.0,0)(2.5,0.866)
\psline[linecolor=black,linestyle=dotted]{-}(-0.2,-0.346)(0,0)
\psline[linecolor=black,linestyle=dotted]{-}(0.8,-0.346)(1,0)
\psline[linecolor=black,linestyle=dotted]{-}(1.8,-0.346)(2,0)
\psline[linecolor=black,linestyle=dotted]{-}(1.2,-0.346)(1,0)
\psline[linecolor=black,linestyle=dotted]{-}(2.2,-0.346)(2,0)
\psline[linecolor=black,linestyle=dotted]{-}(3.2,-0.346)(3,0)
\rput[c](0,0){$\scriptstyle\bullet$}
\rput[c](0.5,0.866){$\scriptstyle\bullet$}
\rput[c](1.0,1.732){$\scriptstyle\bullet$}
\rput[c](1.5,2.598){$\scriptstyle\bullet$}
\rput[c](2.0,1.732){$\scriptstyle\bullet$}
\rput[c](2.5,0.866){$\scriptstyle\bullet$}
\rput[c](3,0){$\scriptstyle\bullet$}
\rput[c](2,0){$\scriptstyle\bullet$}
\rput[c](1,0){$\scriptstyle\bullet$}\rput[b](1.5,1.832){$\scriptstyle{e_1}$}
\rput[r](1.2,1.25){$\scriptstyle{e_2}$}
\rput[l](1.8,1.25){$\scriptstyle{e_3}$}
\rput[t](2,0.800){$\scriptstyle{e_4}$}
\rput[r](0.725,0.400){$\scriptstyle{e_5}$}
\rput[l](0.9,0.432){$\scriptstyle{e_6}$}
\rput[r](1.7,0.400){$\scriptstyle{e_7}$}
\rput[l](2.3,0.400){$\scriptstyle{e_8}$}
\rput[t](2.5,-0.1){$\scriptstyle{e_9}$}
\end{pspicture}
\end{center}

Then working from the left and from the right, we
find that $y(e)\in S$ for all the diagonal edges in the lowest row.

Finally, we get to the middle triangle in the lowest row. For the diagonal 
edges $e$, $e'$ of that triangle, we have found that $y(e),y(e')\in S$, and so
$y(e'')\in S$ for the horizontal edge $e''$ of that triangle too, by Lemma~\ref{lem:thirdedge}.
\end{proof}
If we make the extra assumption that the set of values $d(g.0,0)$, $g\in\Gamma$, is discrete, 
the following result is a consequence of \cite[Theorem~I.8.10]{BridsonHaefliger} 
(applied to the open set $U=\{x\in X:d(x,0)<r_0+\epsilon\}$ for $\epsilon>0$ small).
We shall sketch a proof along the lines of that of Theorem~\ref{thm:mainresult}
which does not make that extra assumption.
\begin{theorem}\label{thm:triplespresentation}With the hypotheses of 
Theorem~\ref{thm:mainresult}, let $S_0=\{g\in S:d(g.0,0)\le2r_0\}$, as before. Then
the set of generators $S_0$, and the relations $g_1g_2g_3=1$, where the $g_i$ 
are each in~$S_0$, give a presentation of~$\Gamma$.
\end{theorem}
\begin{proof}Let $\tilde S_0$ be a set with a bijection $f:\tilde s\mapsto s$ from $\tilde S_0$ to~$S_0$.
Let $F$ be the free group on~$\tilde S_0$, and denote also by~$f$ the induced homomorphism $F\to\Gamma$.
Then $f$ is surjective by Lemma~\ref{lem:tworzero}. Let $\tilde y_1,\ldots,\tilde y_n\in\tilde S_0$, 
and suppose that $\tilde g=\tilde y_1\cdots\tilde y_n\in F$ is in the kernel of~$f$. We must show 
that $\tilde g$ is in the normal closure~$H$ of the set of elements of~$F$ of the form $\tilde s_1\tilde s_2\tilde s_3$,
where $\tilde s_i\in\tilde S_0$ for each~$i$, and $s_1s_2s_3=1$ in~$\Gamma$. As $1\in S_0$, we may
assume that $n$ is a multiple~$3m$ of~3, and form a triangle~$\Delta$ partitioned into $m^2$ congruent
subtriangles, as in the proof of Theorem~\ref{thm:mainresult}. The vertices are again denoted~$v_i$,
and we write $v_{3m}=v_0$, and $y_i=f(\tilde y_i)$ for each~$i$.

We again define a continuous function $\varphi$ from the boundary of~$\Delta$ 
to~$X$ which maps $v_0$ to~$0$ and $v_i$ to $(y_1\cdots y_i).0$ for $i=1,\ldots,3m$, and
which for $i=1,\ldots,3m$ maps the segment $[v_{i-1},v_i]$ to the geodesic from~$\varphi(v_{i-1})$
to~$\varphi(v_i)$. From $y_1y_2\cdots y_n=1$ we see that $\varphi(v_0)=1=g=\varphi(v_{3m})$. This time
the bottom side of~$\Delta$ is mapped in the same way as the other two sides.

Let $\epsilon>0$, with $\epsilon<\delta/2,(r_1-2r_0)/2$, and as before, we partition~$\Delta$ into
subtriangles so that whenever $t,t'$ are in the same subtriangle, $d(\varphi(t),\varphi(t'))<\epsilon$ holds.

Using Lemma~\ref{lem:remepsilon} and the remark after it, we can again choose 
elements $x(v)\in\Gamma$ for each vertex~$v$ of each of the subtriangles, so 
that $d(x(v).0,\varphi(v))\le r_0$ for each~$v$, without the complications about 
the bottom side of~$\Delta$ which had to be dealt with in the proof of 
Theorem~\ref{thm:mainresult}. If $e$ is the edge from $v'$ to~$v$, where $v'$ is
immediately to the left of~$v$, and $v$ is not on the right side of~$\Delta$,
then $y(e)=x(v')^{-1}x(v)\in S_0$. We then define elements 
$y(e)$ for each directed edge~$e$ of each of the subtriangles so that $x(w)=x(v)y(e)$ 
if $e$ is the edge from~$v$ to~$w$. Again using Lemma~\ref{lem:thirdedge}, with
$S$ there replaced by~$S_0$, and arguing as in the conclusion of the proof 
of Theorem~\ref{thm:mainresult}, we deduce that $y(e)\in S_0$ for each edge~$e$.

Of the $r$ edges~$e$ in the segment from~$v_{i-1}$ to~$v_i$, we have $y(e)=1$ 
except for the middle edge, for which $y(e)=y_i$. So $y(e_1)y(e_2)\cdots y(e_{3mr})=1$, 
where $e_1,\ldots,e_{3mr}$ are the successive edges as we traverse the sides of~$\Delta$ 
in a clockwise direction starting at~$v_0$, and $\tilde g=\tilde y(e_1)\cdots\tilde y(e_{3mr})$.

It is now easy to see that $\tilde g\in H$, by using $y(e'')=y(e)y(e')$ in
the situation of~(1), so that $\tilde y(e'')=\tilde y(e)\tilde y(e')h$ for
some $h\in H$. Then in the situation of~(2),
for example, $y(e)y(e')=y(e)(y(e^*)y(e'''))=(y(e)y(e^*))y(e''')=y(e'')y(e''')$,
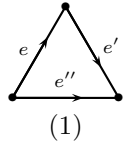
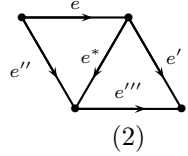
\begin{center}
\begin{pspicture}(-4,-0.5)(6,1.532)
\psset{xunit=0.7,yunit=0.7}%
\psline[linecolor=black]{-}(-4,0)(-3,1.732)(-2,0)(-4,0)
\psline[linecolor=black]{->}(-4,0)(-3.333,1.1547)
\psline[linecolor=black]{->}(-3,1.732)(-2.333,0.577)
\psline[linecolor=black]{->}(-4,0)(-2.666,0)
\rput[c](-4,0){$\scriptstyle\bullet$}
\rput[c](-3,1.732){$\scriptstyle\bullet$}
\rput[c](-2,0){$\scriptstyle\bullet$}
\rput[t](-3,-0.3){(1)}
\psline[linecolor=black]{-}(4,0)(3,1.732)(5,1.732)
\psline[linecolor=black]{-}(4,0)(5,1.732)(6,0)(4,0)
\rput[c](4,0){$\scriptstyle\bullet$}
\rput[c](3.0,1.732){$\scriptstyle\bullet$}
\rput[c](5.0,1.732){$\scriptstyle\bullet$}
\rput[c](6,0){$\scriptstyle\bullet$}
\rput[r](-3.7,0.9){$\scriptstyle{e}$}
\rput[l](-2.3,1){$\scriptstyle{e'}$}
\rput[b](-3,0.2){$\scriptstyle{e''}$}
\psline[linecolor=black]{->}(3,1.732)(3.666,0.577)
\psline[linecolor=black]{->}(3,1.732)(4.333,1.732)
\psline[linecolor=black]{->}(5,1.732)(5.666,0.577)
\psline[linecolor=black]{->}(4,0)(5.333,0)
\psline[linecolor=black]{->}(5,1.732)(4.333,0.577)

\rput[b](4,1.9){$\scriptstyle{e}$}
\rput[l](5.7,1){$\scriptstyle{e'}$}
\rput[r](3.2,0.75){$\scriptstyle{e''}$}
\rput[b](5,0.2){$\scriptstyle{e'''}$}
\rput[r](4.5,1){$\scriptstyle{e^*}$}
\rput[t](5,-0.3){(2)}
\end{pspicture}
\end{center}
so that $\tilde y(e)\tilde y(e')=\tilde y(e)(\tilde y(e^*)\tilde y(e''')h)
=(\tilde y(e)\tilde y(e^*))\tilde y(e''')h=(\tilde y(e'')h')\tilde y(e''')h
=\tilde y(e'')\tilde y(e''')h''$ for some $h,h',h''\in H$. We can, for example, successively 
use this device to remove the right hand strip of triangles from~$\Delta$, reducing
the size of the triangle being treated, and then repeat this process.
\end{proof}
In the next proposition, we use \cite[Theorem~I.8.10]{BridsonHaefliger} to show that under extra
hypotheses (satisfied by the example in Section~3), we can omit the $g\in S_0$ for which 
$d(g.0,0)=2r_0$, and still get a presentation.
\begin{proposition}\label{prop:triplespresentation}Let $X=B(\C^2)$, and suppose that the set of values $d(g.0,0)$, $g\in\Gamma$,
is discrete. Assume also the hypotheses of Theorem~\ref{thm:mainresult}. Then the set
$S_0^*=\{g\in S:d(g.0,0)<2r_0\}$ is a set of generators of~$\Gamma$, and the 
relations $g_1g_2g_3=1$, where the $g_i$ are each in~$S_0^*$,  give a presentation of~$\Gamma$.
\end{proposition}
\begin{proof}For each $g\in S$ such that $d(g.0,0)=2r_0$, let $m$ be the midpoint
of the geodesic from~0 to~$g.0$. Let $M$ be the set of these midpoints. Let $\delta_1>0$ be
so small that $2r_0+2\delta_1<r_1$. Since $M$ and~$S$ are finite, we can choose a positive
$\delta<\delta_1$ so that if $m,m'\in M$ and $g\in S$, and if $d(g.m,m')<2\delta$, then
$g.m=m'$. So if $\gamma,\gamma'\in\Gamma$, $m,m'\in M$ and 
$B(\gamma.m,\delta)\cap B(\gamma'.m',\delta)\ne\emptyset$, then $d(\gamma.0,\gamma'.0)<2r_0+2\delta<r_1$,
so that $\gamma^{-1}\gamma'\in S$ by Theorem~\ref{thm:mainresult}, and $d(m,\gamma^{-1}\gamma'.m')<2\delta$,
so that $\gamma.m=\gamma'.m'$ by choice of~$\delta$. Thus if $Y$ is the union of the $\Gamma$-orbits of
all $m\in M$, then the balls $B(y,\delta)$, $y\in Y$, are pairwise disjoint. Now let $X'$ denote the 
subset of~$X$ obtained by removing all these balls. It follows (because the ambient 
dimension is $>2$) that $X'$~is still simply connected.

Let $U=\{x\in X' : d(x,0)<r_0+\delta'\}$ of~$X'$, for some $\delta'>0$. The proposition will follow
from~\cite[Theorem~I.8.10]{BridsonHaefliger}, applied to this~$U$, once we show that if $\delta'>0$ is small 
enough, then any $g\in\Gamma$ such  that $g(U)\cap U \ne \emptyset$ must satisfy $d(g.0,0)<2r_0$. 
Clearly $d(g.0,0) < 2r_0+2\delta'$ holds, and because of the discreteness hypothesis, if $\delta'>0$
is small enough, one even has $d(g.0,0) \le 2r_0$. Now suppose $d(g.0,0)=2r_0$, 
let $m$~be the midpoint on the geodesic from~$0$ to~$g.0$, and suppose 
$x\in g(U)\cap U$.  Then $d(0,x)<r_0+\delta'$ and $d(g.0,x)<r_0+\delta'$.
Using the CAT(0) property of~$X$, this shows that
$d(x,m)^2+r_0^2 \leq (r_0+\delta')^2$,
so that
\begin{displaymath}
d(x,m) \le \sqrt{(r_0+\delta')^2-r_0^2},
\end{displaymath}
and this last can be made less than~$\delta$ if $\delta'$~is chosen
small enough.  This contradicts the hypothesis that $x\in U\subset X'$.
\end{proof}

The fact that $X'$~is simply connected could also be used to modify
the version of the proof going through the triangle-shaped simplicial
complex.

\end{section}
\begin{section}{An example}
Let $\ell=\Q(\zeta)$, where $\zeta$ is a primitive 12-th root of unity. Then
$\zeta^4=\zeta^2-1$, so that $\ell$ is a degree~4 extension of~$\Q$. Let
$r=\zeta+\zeta^{-1}$ and $k=\Q(r)$. Then $r$ and~$\zeta^3$ are square roots 
of~3 and~$-1$, respectively, and if $\zeta=e^{2\pi i/12}$ then $r=+\sqrt{3}$ and $\zeta^3=i$.
Let
\begin{equation}\label{eq:Fdefn}
F=\begin{pmatrix}-r-1&1&0\\1&1-r&0\\0&0&1\end{pmatrix},
\end{equation}
and form the group
\begin{displaymath}
\Gamma=\{g\in M_{3\times3}(\Z[\zeta]):g^*Fg=F\}/\{\zeta^\nu I:\nu=0,1,\ldots,11\}.
\end{displaymath}
We shall use the results of Section~\ref{sec:generalresults} to find a presentation for $\Gamma$.
Let us first motivate the choice of this example. Now $\iota(g)=F^{-1}g^*F$ defines an involution of the second kind on the simple algebra $M_{3\times3}(\ell)$.
We can define an algebraic group $G$ over~$k$ so that
\begin{displaymath}
G(k)=\{g\in M_{3\times3}(\ell):\iota(g)g=I\ \text{and}\ \det(g)=1\}.
\end{displaymath}
For the corresponding adjoint group $\overline G$, 
\begin{displaymath}
\overline G(k)=\{g\in M_{3\times3}(\ell):\iota(g)g=I\}/\{tI:t\in\ell\ \text{and}\ \bar tt=1\}.
\end{displaymath}
Now $k$ has two archimedean places $v_+$ and $v_-$, corresponding to the two embeddings 
$k\hookrightarrow\R$ mapping $r$ to $+\sqrt3$ and~$-\sqrt3$, respectively.
The eigenvalues of~$F$ are 1 and $-r\pm\sqrt{2}$. So the form~$F$ is definite for $v_-$
but not for $v_+$. Hence
\begin{displaymath}
\overline G(k_{v_-})\cong PU(3)\quad\text{and}\quad\overline G(k_{v_+})\cong PU(2,1).
\end{displaymath} 
Letting $V_f$ denote the set of non-archimedean places of~$k$, if $v\in V_f$, then either
$\overline G(k_v)\cong PGL(3,k_v)$ if $v$ splits in~$\ell$, or
$\overline G(k_v)\cong PU_F(3,k_v(i))$ if $v$ does not split in~$\ell$.
With a suitable choice of maximal parahorics $\overline P_v$ in~$\overline G(k_v)$,
the following group
\begin{displaymath}
\bar\Gamma=\overline G(k)\cap \prod_{v\in V_f}\overline P_v
\end{displaymath}
is one of the maximal arithmetic subgroups of $PU(2,1)$
whose covolume has the form $1/N$, $N$ an integer. Prasad and Yeung \cite{PY1,PY2}
have described all such subgroups, up to $k$-equivalence. In this case $N=864$. 
As in~\cite{CS}, lattices can be used to describe concretely maximal parahorics.
We can take $\overline P_v=\{g\in\overline G(k_v):g.x_v=x_v\}$, where, in the cases when~$v$
splits in~$\ell$, $x_v$ is the
homothety class of the $\mathfrak o_v$-lattice $\mathfrak o_v^3\subset k_v^3$, where
$\mathfrak o_v$ is the valuation ring of~$k_v$. When $v$ does not split, $x_v$ is the
lattice $\mathfrak o_v^3$, where now $\mathfrak o_v$ is the valuation ring of $k_v(i)$.
With this particular choice of parahorics, $\bar\Gamma$ is just~$\Gamma$.

The action of~$\Gamma$ on the unit ball $X=B(\C^2)$ in~$\C^2$ is described as follows, 
making explicit the isomorphism $\overline G(k_{v_+})\cong PU(2,1)$. Let
\begin{equation}\label{eq:variousmatrices}
\gamma_0=\begin{pmatrix}
1&  0&0\\
1&1-r&0\\
0&  0&1
\end{pmatrix},
\
F_{\mathrm{diag}}=\begin{pmatrix}
1&0&  0\\
0&1&  0\\
0&0&1-r
\end{pmatrix},
\ \text{and}
\
F_0=\begin{pmatrix}
1&0&  0\\
0&1&  0\\
0&0&-1
\end{pmatrix}.
\end{equation}
Then $\gamma_0^tF_{\mathrm{diag}}\gamma_0=(1-r)F$, so a matrix $g$ is unitary with
respect to~$F$ if and only if $g'=\gamma_0g\gamma_0^{-1}$ is unitary with respect
to~$F_{\mathrm{diag}}$. Now let $D$ be the diagonal matrix
with diagonal entries 1, 1 and~$\sqrt{\sqrt{3}-1}$. Taking $r=+\sqrt3$, 
if $g'$ is unitary with respect to~$F_{\mathrm{diag}}$, then $\tilde g=Dg'D^{-1}$ is unitary
with respect to~$F_0$; that is, $\tilde g\in U(2,1)$. If $Z=\{\zeta^\nu I:\nu=0,\ldots,11\}$,
for an element $gZ$ of~$\Gamma$, the action of $gZ$ on~$B(\C^2)$ is given by 
the usual action of~$\tilde g$. That is,
\begin{displaymath}
(gZ).(z,w)=(z',w')\quad\text{if}\quad 
\tilde g\begin{pmatrix}z\\w\\1\end{pmatrix}=\lambda\begin{pmatrix}z'\\w'\\1\end{pmatrix}
\quad\text{for some}\ \lambda\in\C.
\end{displaymath}

Now let $u$ and $v$ be the matrices
\begin{displaymath}
u=\begin{pmatrix}
\zeta^3+\zeta^2-\zeta&1-\zeta&0\\
\zeta^3+\zeta^2-1&\zeta-\zeta^3&0\\
0&0&1
\end{pmatrix}
\quad\text{and}\quad
v=\begin{pmatrix}
\zeta^3&0&0\\
\zeta^3+\zeta^2-\zeta-1&1&0\\
0&0&1\\
\end{pmatrix},
\end{displaymath}
respectively. They have entries in~$\Z[\zeta]$, are unitary with respect to~$F$,
and satisfy 
\begin{displaymath}
u^3=I,\ v^4=I,\ \text{and}\ (uv)^2=(vu)^2. 
\end{displaymath}
They (more precisely, $uZ$ and $vZ$) generate a subgroup $K$ of~$\Gamma$ of order~288.
Magma shows that an abstract group with presentation $\langle u,v:u^3=v^4=1,\ (uv)^2=(vu)^2\rangle$
has order~288, and so $K$ has this presentation. 

Let us write simply~0 for the origin~$(0,0)\in B(\C^2)$, and $g.0$ in place of $(gZ).(0,0)$.

\begin{lemma} For the action of $\Gamma$ on $X$, $K$ is the stabilizer of~0.
\end{lemma}
\begin{proof}It is easy to see that $gZ\in\Gamma$ fixes~0 if and only if $gZ$ has a matrix representative
\begin{displaymath}
\begin{pmatrix}
g_{11}&g_{12}&0\\
g_{21}&g_{22}&0\\
0&0&1
\end{pmatrix},
\end{displaymath}
for suitable $g_{ij}\in\Z[\zeta]$. Since the $2\times2$ block matrix in the upper left 
of~$F$ is definite when $r=+\sqrt3$, it is now routine to determine all such $g_{ij}$.
\end{proof}
The next step is to find $g\in\Gamma\setminus K$ for which $d(g.0,0)$ is small, 
where $d$ is the hyperbolic metric on~$B(\C^2)$. Now
\begin{equation}\label{eq:hyperbolicdistance}
\cosh^2(d(z,w))=\frac{|1-\langle z,w\rangle|^2}{(1-|z|^2)(1-|w|^2)},
\end{equation}
(see \cite[Page~310]{BridsonHaefliger} for example) 
where $\langle z,w\rangle=z_1\bar w_1+z_2\bar w_2$ and $|z|=\sqrt{|z_1|^2+|z_2|^2}$
for $z=(z_1,z_2)$ and~$w=(w_1,w_2)$ in~$B(\C^2)$.

In particular, writing 0 for the origin in~$B(\C^2)$, and using 
$g.0=(g_{13}/g_{33},g_{23}/g_{33})$ and
$|g_{13}|^2+|g_{23}|^2=|g_{33}|^2-1$ for $g=(g_{ij})\in U(2,1)$, we see that
\begin{equation}\label{eq:distg0to0}
\cosh^2(d(0,g.0))=|g_{33}|^2
\end{equation}
for $g\in U(2,1)$. Notice that for $gZ\in\Gamma$, the $(3,3)$-entry of~$g$ is equal to the $(3,3)$-entry
of the $\tilde g\in U(2,1)$ defined above, and so \eqref{eq:distg0to0} holds also
for $gZ\in\Gamma$.

The matrix 
\begin{displaymath}
b=\begin{pmatrix}
                         1&                     0&0\\
-2\zeta^3-\zeta^2+2\zeta+2&\zeta^3+\zeta^2-\zeta-1&-\zeta^3-\zeta^2\\
             \zeta^2+\zeta&             -\zeta^3-1&-\zeta^3+\zeta+1
\end{pmatrix}
\end{displaymath}
is unitary with respect to~$F$. We shall see below that $u$, $v$ and~$b$
generate~$\Gamma$, and use the results of Section~\ref{sec:generalresults} to show that 
some relations they satisfy give a presentation of~$\Gamma$. This $b$ was found 
by a computer search for $g\in\Gamma\setminus K$ for which $d(g.0,0)$ is small.

Notice that $d(g.0,0)$ is constant on each double coset $KgK$. Calculations 
using~\eqref{eq:distg0to0} showed that amongst the 288 elements $g\in\Gamma$
of the form $bkb$, $k\in K$, there are ten different values of $d(g.0,0)$. Representatives
$\gamma_j$ of the 20~double cosets $KgK$ in which the $g\in bKb$ lie were chosen. The
smallest few $|(\gamma_j)_{33}|^2$ and the corresponding $d(\gamma_j.0,0)$ (rounded 
to 4~decimal places) are as follows:
\begin{center}
\begin{tabular}[t]{|c|c|c|c|}\hline
\vbox to 2.5ex{}$j$&$\gamma_j$&$|(\gamma_j)_{33}|^2$&$d(\gamma_j.0,0)$\\[0.25ex]\hline
\vbox to 2.5ex{}$1$&$1$&1&0\\[0.25ex]\hline
\vbox to 2.5ex{}$2$&$b$&$\sqrt{3}+2$&$1.2767$\\[0.25ex]\hline
\vbox to 2.5ex{}$3$&$bu^{-1}b$&$2\sqrt{3}+4$&$1.6629$\\[0.25ex]\hline
\vbox to 2.5ex{}$4$&$bu^{-1}v^{-1}u^{-1}b$&$3\sqrt{3}+6$&$1.8778$\\[0.25ex]\hline
%
%
%
%
\end{tabular}
\end{center}

\medskip\noindent We use Theorem~\ref{thm:mainresult} to show that our computer search has
not missed any $g\in\Gamma$ for which $d(g.0,0)$ is small. Let $S=K\cup K\gamma_2K\cup K\gamma_3K$,
consisting of the $g\in\Gamma$ found by the computer search to satisfy
$|g_{33}|^2\le2\sqrt{3}+4$. To verify that $S$ generates~$\Gamma$,
we need to make a numerical estimate. A somewhat longer direct proof that $\langle S\rangle=\Gamma$ 
is given in the next section.

\begin{proposition}\label{prop:volestimate}For the given $S\subset\Gamma$, the normalized hyperbolic volume 
$\vol(\cF_S)$ of $\cF_S=\{x\in B(\C^2):d(x,0)\le d(x,g.0)\ \text{for all}\ g\in S\}$ satisfies
$\vol(\cF_S)<2/864$. The set $S$ generates~$\Gamma$.
\end{proposition}
\begin{proof}Standard numerical integration methods show that (up to several decimal place
accuracy) $\vol(\cF_S)$ equals~$1/864$, but all we need is that $\vol(\cF_S)<2/864$.  
Let us make explicit the normalization of hyperbolic volume element in~$B(\C^2)$ which makes 
the formula $\chi(\Gamma\backslash B(\C^2))=3\vol(\cF)$ true. For $z\in B(\C^2)$, write 
\begin{displaymath}
\begin{aligned}
t&=\tanh^{-1}|z|=\frac{1}{2}\log\Bigl(\frac{1+|z|}{1-|z|}\Bigr)=d(0,z),\\
\Theta&=z/|z|\in S^1(\C^2),\\
d\Theta&=\text{the usual measure on}\ S^1(\C^2),\ \text{having total volume}\ 2\pi^2,\\
d\vol(z)&=\frac{2}{\pi^2}\sinh^3(t)\cosh(t)dt\,d\Theta.
\end{aligned}
\end{displaymath}
Let $\Gamma'=\langle S\rangle$. Then
the Dirichlet fundamental domains~$\cF$ and~$\cF'$ of~$\Gamma$ and~$\Gamma'$ satisfy
$\cF\subset\cF'\subset\cF_S$. By \cite[\S8.2, the $\cC_{11}$ case]{PY1}, $\vol(\cF)=1/864$. Let $M=[\Gamma:\Gamma']$. Then
$\vol(\cF')=M\vol(\cF)$ (and $\vol(\cF')=\infty$ if $M=\infty$), so that $M\vol(\cF)=\vol(\cF')\le\vol(\cF_S)<2\vol(\cF)$
implies that $M=1$ and $\Gamma'=\Gamma$.
\end{proof}

\begin{proposition}\label{prop:rzeroestimate}For the given $S\subset\Gamma$, the value of $r_0=\sup\{d(x,0):x\in \cF_S\}$
is~$\frac{1}{2}d(\gamma_3.0,0)=\frac{1}{2}\cosh^{-1}(\sqrt{3}+1)$.
\end{proposition}
\begin{proof}We defer the proof of this result to the next section.
\end{proof}
One may verify that if $g,g'\in S$ and $gg'\not\in S$, then $|(gg')_{33}|^2\ge3\sqrt{3}+6$.
So if $r_1$ satisfies $2\sqrt{3}+4<\cosh^2(r_1)<3\sqrt{3}+6$, then~$S$
satisfies the conditions of Theorem~\ref{thm:mainresult}. 
\begin{remark}By Corollary~\ref{cor:funddomain}, $\cF_S=\cF$, so that $\vol(\cF_S)$ 
equals~$1/864$.
\end{remark}
Having verified all the conditions of Theorem~\ref{thm:mainresult}, we now know that $S$ contains
all elements $g\in\Gamma$ satisfying $d(g.0,0)\le r_1$. The double 
cosets $K\gamma_jK$, $j=1,2,3$, are symmetric because $(buvu)^2v=\zeta^{-1}I$, and $(bu^{-1})^4=I$ (see below). 
The sizes of these $K\gamma_jK$ are~288, $288^2/4$ and~$288^2/3$, respectively, 
adding up to~$48,672$, because $\{k\in K:\gamma_2^{-1}k\gamma_2\in K\}=\langle v\rangle$, and
$\{k\in K:\gamma_3^{-1}k\gamma_3\in K\}=\langle u\rangle$. Proposition~\ref{prop:triplespresentation} 
gives a presentation of~$\Gamma$ which we now simplify. 
\begin{proposition}\label{prop:c11emptypresentation}A presentation of~$\Gamma$ is given by the generators
$u$, $v$ and~$b$ and the relations
\begin{equation}\label{eq:c11emptypresentation}
u^3=v^4=b^3=1,\ (uv)^2=(vu)^2,\ vb=bv,\ (buv)^3=(buvu)^2v=1.
\end{equation}
\end{proposition}
\begin{proof}By Proposition~\ref{prop:triplespresentation}, the set $S^*_0=\{g\in\Gamma:d(g.0,0)<2r_0\}$ 
of generators, and the relations $g_ig_jg_k=1$, where $g_1,g_2,g_3\in S^*_0$, give a presentation of~$\Gamma$.
By Theorem~\ref{thm:mainresult}, $S_0^*$ is the union of the two double cosets $K\gamma_1K=K$ and~$K\gamma_2K$.
So these relations have the form $(k_1'\gamma_{i_1}k_1'')(k_2'\gamma_{i_2}k_2'')(k_3'\gamma_{i_3}k_3'')=1$, 
where $k_\nu',k_\nu''\in K$ and $i_1,i_2,i_3\in\{1,2\}$. Using the known presentation of~$K$, and cyclic permutations,
the relations of the form $\gamma_{i_1}k_1\gamma_{i_2}k_2\gamma_{i_3}k_3=1$, where $1\le i_1\le i_2,i_3\le2$, 
are sufficient to give a presentation. After finding a word in~$u$ and $v$ for each such $k_i$, 
we obtained a list of words in $u$, $v$ and~$b$ coming from these relations. Magma's routine {\tt Simplify}
can be used to complete the proof, but it is not hard to see this more directly as follows.
When $i_1=1$, we need only include in the list of words all those coming from the relations of the form (i):~$bkb=k'$,
for $k,k'\in K$. When $i_1=2$ (so that $i_2=i_3=2$ too), we need only include, for each $k_1\in K$ such that
$bk_1b\in KbK$, a single pair $(k_2,k_3)$ such that $bk_1bk_2bk_3=1$, since the other such relations
follow from this one and the relations~(i). The only relations of the form~(i) are the relations $bv^\nu uvub=v^{\nu-1}u^{-1}v^{-1}u^{-1}$,
$\nu=0,1,2,3$,
which follow from the relations $vb=bv$ and~$(buvu)^2v=1$ in~\eqref{eq:c11emptypresentation}. Next, matrix calculations show
that there are 40 elements~$k\in K$ such that $bkb\in KbK$, giving 40 relations $bkb=k'bk''$. We need only show that all of these
are deducible from the relations given in~\eqref{eq:c11emptypresentation}.
Using $bv=vb$, any equation $bkb=k'bk''$ gives equations $bv^ikv^jb=v^ik'bk''v^j$.
So we needed only deduce from~\eqref{eq:c11emptypresentation} the five relations
\begin{displaymath}
\begin{aligned}
b1b&=uvuvbuvu,\quad bub=ubu,\quad buvu^{-1}b=v^{-1}u^{-1}v^{-1}bu^{-1},\\ 
buvub&=(vu)^{-2},\quad\text{and}\quad buvu^{-1}vub=v^2u^{-1}v^{-1}u^{-1}bu^{-1}v^{-1}u^{-1}.
\end{aligned}
\end{displaymath}
Firstly, $(buvu)^2v=1$ and $(vu)^2=(uv)^2$ imply that $b^{-1}=vuvubuvu$, and this
and~$b^3=1$ imply the first relation.  The relations
$(bvu)^3=1$ and $bv=vb$ imply that $bub=(vubvuv)^{-1}$, and this and $b^{-1}=vuvubuvu$ give
$bub=ubu$. To get the third relation, use $vb=bv$ to see that $v(buvu^{-1}b)u$ equals
\begin{displaymath}
b(vuvu)ubu=b(vu)^2bub=b(vu)^2b(vu)^2(vu)^{-2}ub=v(uv)^{-2}u=u^{-1}v^{-1}b.
\end{displaymath}
The fourth relation is immediate from $(buvu)^2v=1$ and $bv=vb$. Finally, from
$(uv)^2=(vu)^2$ and our formula for $b^{-1}$ we have
\begin{displaymath}
(uvu)^{-1}b^{-1}(uvu)=(uvu)^{-1}(uvuvbuvu)(uvu)=vbuvu^2vu=vbk
\end{displaymath}
for $k=uvu^{-1}vu$, using $u^3=1$. So $vbk$ has order~3. Hence $bkb=v^{-1}(kvbv)^{-1}v^{-1}$,
and the fifth relation easily follows.
\end{proof}

As mentioned above, $(bu^{-1})^4=1$. By Proposition~\ref{prop:c11emptypresentation}, 
this is a consequence of the relations in~\eqref{eq:c11emptypresentation}. Explicitly,
\begin{displaymath}
\begin{aligned}
(bu^{-1})^4=bu(ubu)(ubu)(ubu)u=bu(bub)(bub)(bub)u&=b(ubu)bbubb(ubu)\\
&=b(bub)bbubb(bub)\\
&=bbuuub\\
&=1.
\end{aligned}
\end{displaymath}

Let us record here the connection between~$\Gamma$ and a Deligne-Mostow group whose presentation (see Parker~\cite{Parker}) is
\begin{displaymath}
\Gamma_{3,4}=\langle J,R_1,A_1\ :\ J^3=R_1^3=A_1^4=1,\ A_1=(JR_1^{-1}J)^2,\ A_1R_1=R_1A_1\rangle.
\end{displaymath}
Using the fact that the orbifold Euler characteristics of $\Gamma_{3,4}\backslash B(\C^2)$ and
$\Gamma\backslash B(\C^2)$ are both equal to~$1/288$, several experts, including John Parker,
Sai-Kee Yeung and Martin Deraux, knew that $\Gamma$ and~$\Gamma_{3,4}$ are isomorphic.
The following proof based on presentations is a slight modification of one communicated to 
us by~John Parker. It influenced our choice of the generators $u$ and~$v$ for~$K$.

\begin{proposition}\label{prop:DMiso}There is an isomorphism $\psi:\Gamma\to\Gamma_{3,4}$ such that
\begin{equation}\label{eq:psiconds}
\psi(u)=JR_1J^{-1},\quad \psi(v)=A_1,\quad \text{and}\quad \psi(b)=R_1.
\end{equation}
Its inverse satisfies $\psi^{-1}(J)=buv$, $\psi^{-1}(R_1)=b$ and $\psi^{-1}(A_1)=v$. 
\end{proposition}
\begin{proof}Setting $R_2=JR_1J^{-1}$, we have $R_1R_2A_1=J$ and $A_1JR_2=JR_1^{-1}J$. So
\begin{displaymath}
(\psi(u)\psi(v))^2=(R_1^{-1}J)^2=R_1^{-1}\cdot A_1JR_2=A_1R_1^{-1}JR_2=A_1R_2A_1R_2=(\psi(v)\psi(u))^2.
\end{displaymath}
Next, $\psi(b)\psi(u)\psi(v)=R_1R_2A_1=J$, which implies that $(\psi(b)\psi(u)\psi(v))^3=1$.
Now $(\psi(b)\psi(u)\psi(v)\psi(u))^2\psi(v)=(R_1R_2A_1R_2)^2A_1=(JR_2)^2A_1=JR_2JR_1^{-1}J=1$. So there
is a homomorphism $\psi:\Gamma\to\Gamma_{3,4}$ satisfying~\eqref{eq:psiconds}. We similarly check that
we have a homomorphism~$\tilde\psi:\Gamma_{3,4}\to\Gamma$ mapping $J$, $R_1$ and~$A_1$ to $buv$, $b$ and~$v$,
respectively, and that $\psi$ and $\tilde\psi$ are mutually inverse.
\end{proof}
We now exhibit a torsion-free subgroup of $\Gamma$ having index~864. It has three generators,
all in $KbK$. The elements of~$K$ are most neatly expressed if we use not only the generators $u$ and~$v$,
but also $j=(uv)^2$, which is the diagonal matrix with diagonal entries $\zeta$, $\zeta$ and~1, 
and which generates the center of~$K$.

\begin{lemma}\label{lem:torsionelts}The non-trivial elements of finite order in~$\Gamma$ 
have order dividing~24.
\begin{itemize}
\item[(i)] Any element of order~2 is conjugate to one of $v^2$, $j^6$ or $(bu^{-1})^2$.
\item[(ii)] Any element of order~3 is conjugate to one of $u$, $j^4$, $uj^4$, $buv$, or their inverses.
\end{itemize}
\end{lemma}
\begin{proof}By \cite[Corollary II.2.8(1)]{BridsonHaefliger} for example, any $g\in U(2,1)$ of finite order fixes at least one point of~$B(\C^2)$, and so
in particular this holds for any~$g\in\Gamma$ of finite order. Conjugating~$g$, we may assume that
the fixed point is in the fundamental domain~$\cF$ of~$\Gamma$, and so $d(g.0,0)\le2r_0$.
Thus $g$ lies in $K\cup K\gamma_2K\cup K\gamma_3K$, and so is conjugate 
to an element of~$K\cup \gamma_2K\cup \gamma_3K$. Checking these 864 elements, we see that $g$ must
have order dividing~24. After listing the elements of order 2 and~3 amongst them,
routine calculations verify (i) and~(ii). 
\end{proof}

\begin{proposition}\label{prop:index864group}The elements
\begin{displaymath}
a_1=vuv^{-1}j^4buv j^2,\quad
a_2=v^2u buv^{-1}uv^2j\quad\text{and}\quad
a_3=u^{-1}v^2u j^9bv^{-1}uv^{-1}j^8
\end{displaymath}
generate a torsion-free subgroup~$\Pi$ of~864, for which $\Pi/[\Pi,\Pi]\cong\Z^2$.
\end{proposition}
\begin{proof}Using our presentation of~$\Gamma$, Magma's {\tt Index\/} command
verifies that $\Pi$ has index~864 in~$\Gamma$, and the 
{\tt AbelianQuotientInvariants\/} command verifies that it has
abelianization~$\Z^2$. 

We now check that $\Pi$ is torsion-free. Suppose that $\Pi$ contains
an element $\pi\ne1$ of finite order. By Lemma~\ref{lem:torsionelts},
we can assume that $\pi$ has order~2 or~3. So for one of the
elements~$t$ listed in (i) and~(ii) of Lemma~\ref{lem:torsionelts},
there is a $g\in\Gamma$ such that $gtg^{-1}\in\Pi$. Using~{\tt Index},
one verifies that the 864 elements $b^\mu k$, $\mu=0,1,2$, $k\in K$, form a
transversal for $\Pi$ in~$\Gamma$. So we may assume that $g=b^\mu k$. But
now {\tt Index\/} verifies that none of the elements
$b^\mu kt(b^\mu k)^{-1}$ is in~$\Gamma$.
\end{proof}

We conclude this section by mentioning some other properties of~$\Pi$. 

Let us first note that $\Pi$ cannot be lifted to a subgroup of~$SU(2,1)$. The 
determinants of $a_1$, $a_2$ and~$a_3$ are $\zeta^3$, $\zeta^3$
and~$-1$, respectively, and so the $a_\nu$ could be replaced by $\zeta^{-1}\omega^{i_1}a_1$,
$\zeta^{-1}\omega^{i_2}a_2$ and $-\omega^{i_3}a_3$, where $\omega=e^{2\pi i/3}$
and $i_1,i_2,i_3\in\Z$, to obtain generators with determinant~1. But
$a_2^{-3}a_3^3a_1a_2a_3^{-3}a_2^3a_3^{-1}a_1^{-1}a_2^{-1}a_1a_3a_1^{-1}=\zeta^{-4}I$
is unchanged by any choice of the integers~$i_1$, $i_2$ and~$i_3$, as the number of $a_\nu$'s appearing 
in the product on the left is
equal to the number of~$a_\nu^{-1}$'s, for each~$\nu$. So we get a relation in $PU(2,1)$
but not in~$SU(2,1)$. It was found using Magma's {\tt Rewrite\/} command, 
which derives a presentation of~$\Pi$ from that of~$\bar\Gamma$.

Magma shows that the normalizer of~$\Pi$ in~$\Gamma$ contains~$\Pi$ as
a subgroup of index~3, and is generated by~$\Pi$ and~$j^4$. One may verify that
\begin{displaymath}
\begin{aligned}
j^4a_1j^{-4}&=\zeta^3a_3a_2^{-3}a_3^3a_1,\\
j^4a_2j^{-4}&=\zeta^{-1}a_3^{-1},\quad\text{and}\\
j^4a_3j^{-4}&=\zeta^{-1}a_1^{-1}a_2^{-1}a_1a_2^2a_1^{-1}a_2^{-1}a_1a_3^{-1}a_1^{-1}a_2a_1.\\
\end{aligned}
\end{displaymath}
Let us show that $j^4$ induces a non-trivial action on~$\Pi/[\Pi,\Pi]$.
By Proposition~\ref{prop:index864group}, there is an isomorphism $\varphi:\Pi/[\Pi,\Pi]\to\Z^2$, so we have
a surjective homomorphism $f:\Pi\to\Pi/[\Pi,\Pi]\cong\Z^2$. Using the relation 
$a_2^2a_1^{-1}a_2^{-1}a_1a_3^3a_1a_2^{-3}a_3^3a_1a_3a_1=\zeta^3I$, we see that $3f(a_1)-2f(a_2)+7f(a_3)=(0,0)$,
and since $\Pi/[\Pi,\Pi]\cong\Z^2$, this must be the only condition on the $f(a_\nu)$. So we can choose
the isomorphism $\varphi$ so that $f$ maps $a_1$, $a_2$ and~$a_3$ to $(1,3)$, $(-2,1)$ and
$(-1,-1)$, respectively, and then
\begin{equation}\label{eq:j4actiononabelianization}
f(\pi)=(m,n)\quad\implies f(j^4\pi j^{-4})=(m,n)\begin{pmatrix}0&-1\\1&-1\end{pmatrix}\quad\text{for all}\ \pi\in\Pi.
\end{equation}

Next consider the ball quotient $X=\Pi\backslash B(\C^2)$. Now $j^4$ induces an automorphism
of~$X$. Let us show that this automorphism has precisely 9 fixed points.

\begin{proposition}\label{prop:sigmafixedpts}The automorphism of~$X$ induced by~$j^4$ has exactly 9 fixed points. These are
the three points $\Pi(b^\mu.0)$, $\mu=0,1,-1$, and six points $\Pi(h_i.z_0)$, where $h_i\in\Gamma$ for $i=1,\ldots,6$,
and where $z_0\in B(\C^2)$ is the unique fixed point of~$buv$.
\end{proposition}
\begin{proof}If $\Pi(j^4.z)=\Pi z$, then $\pi j^4.z=z$ for some $\pi\in\Pi$. This implies that $\pi j^4$ has finite order. It cannot be trivial,
since $\Pi$ is torsion-free. If $\pi\in\Pi$, then $\pi'=(\pi j^4)^3=(\pi)(j^4\pi j^8)(j^8\pi j^4)$ is also
in~$\Pi$. Since the possible orders of the elements of~$\Gamma$ 
are the divisors of~24, if $\pi j^4$ has finite order, then $1=(\pi j^4)^{24}=(\pi')^8$,
so $\pi'$ must be~1, so that $(\pi j^4)^3$ must be~1. So $\pi j^4$ must have order~3.
So for one of the eight elements~$t$ listed in Lemma~\ref{lem:torsionelts}(ii), 
$\pi j^4=gtg^{-1}$ for some $g\in\Gamma$. Thus $gtg^{-1}j^{-4}\in\Pi$.
Since the elements $b^\mu k$, $\mu=0,1,-1$ and $k\in K$, form a set of
coset representatives for $\Pi$ in~$\Gamma$, and since $j^4$
normalizes $\Pi$, we can assume that $g=b^\mu k$ for some $\mu$ and~$k$.

When $t=j^4$, we have $b^\mu k t k^{-1}b^{-\mu}j^{-4}
=b^\mu j^4b^{-\mu}j^{-4}$, independent of~$k$. We find that these three elements are in~$\Pi$. Explicitly,
$b^\mu j^4b^{-\mu}j^{-4}=\pi_\mu$ for
\begin{equation}\label{eq:bmuOfixedpts}
\pi_0=1,\ \pi_1=\zeta^{-4}a_2a_1^{-2}a_3^{-3}a_1^{-1}\ \text{and}\ \pi_{-1}=a_2^2a_1a_3a_1^{-1}.
\end{equation}
and these equations mean that the three points $\Pi(b^\mu.0)$ are fixed by~$j^4$.

Write $\pi j^4=gtg^{-1}$ for some $g\in\Gamma$, where $t$ is  one of the eight elements~$t$ listed in Lemma~\ref{lem:torsionelts}(ii). In the
notation of~\eqref{eq:bmuOfixedpts}, and writing $g=\pi'b^\mu k$, where $\pi'\in\Pi$, $\mu\in\{0,1,-1\}$, and $k\in K$,
we get
\begin{displaymath}
\begin{aligned}
\pi j^4=\pi'b^\mu ktk^{-1}b^{-\mu}{\pi'}^{-1}
&=\pi'b^\mu ktk^{-1}(j^{-4}b^{-\mu}\pi_\mu j^4){\pi'}^{-1}\\
&=\pi'(b^\mu k)(tj^{-4})(b^\mu k)^{-1}(\pi_\mu j^4{\pi'}^{-1}j^{-4})j^4.
\end{aligned}
\end{displaymath}
So $(b^\mu k)(tj^{-4})(b^\mu k)^{-1}$ is in~$\Pi$, and therefore 
either $t=j^4$ or $tj^{-4}$ has infinite order. In particular, apart from $t=j^4$,
our $t$ cannot be in~$K$, and so must be $buv$ or~$(buv)^{-1}$.

We find that $b^\mu k t k^{-1}b^{-\mu}j^{-4}\in\Pi$ never occurs when $t=(buv)^{-1}$. 
For $t=buv$, we find that $b^\mu k t k^{-1}b^{-\mu}j^{-4}\in\Pi$ for
only 18~pairs $(\mu,k)$. This means that $j^4$
fixes $\Pi(b^\mu k.z_0)$ for these 18 $(\mu,k)$'s. If $(\mu,k)$ satisfies 
$b^\mu k t k^{-1}b^{-\mu}j^{-4}\in\Pi$, then so does $(\mu,kj^4)$, since we 
can write $b^\mu j^4=\pi_\mu j^4b^\mu$ for some $\pi_\mu\in\Pi$, as we have just seen.
Moreover, $\Pi(b^\mu kj^4.z_0)=\Pi(b^\mu k.z_0)$, since $kj^4=j^4k$ and so
\begin{displaymath}
\Pi(b^\mu kj^4.z_0)=\Pi(\pi_\mu j^4b^\mu k.z_0)=\Pi(j^4b^\mu k.z_0)=\Pi(b^\mu k.z_0).
\end{displaymath}
So we need only consider six of the $(\mu,k)$'s, 
and correspondingly setting
\begin{displaymath}
\begin{aligned}
h_1&=b^{-1}vu j^3,\\
h_4&=b^{-1}v^2u j^3,\\
\end{aligned}
\quad
\begin{aligned}
h_2&=u^{-1}v j,\\
h_5&=v j^2,\\
\end{aligned}
\quad
\begin{aligned}
h_3&=buv^2j^2,\\
h_6&=bvu^{-1}v,\\
\end{aligned}
\end{displaymath}
we have $h_i(buv)h_i^{-1}j^{-4}=\pi_i'\in\Pi$ for $i=1,\ldots,6$; explicitly,
\begin{displaymath}
\begin{aligned}
\pi_1'&=\zeta^4a_2^2a_1a_3^3,\\
\pi_4'&=\zeta^{-5}a_3^3a_1^2a_3^3,\\
\end{aligned}
\quad
\begin{aligned}
\pi_2'&=j^8a_1j^4,\\
\pi_5'&=\zeta^{-1}j^4a_1^{-1}a_2^{-1}j^8,\\
\end{aligned}
\quad
\begin{aligned}
\pi_3'&=\zeta^2j^8a_1a_2^3j^4a_2a_1a_2^{-2}a_1^{-1}.\\
\pi_6'&=\zeta a_2a_1^{-1}.\\
\end{aligned}
\end{displaymath}
The six points $\Pi(h_i.z_0)$ are distinct, as we see by checking that 
(a) the nontrivial $g\in\Gamma$ fixing~$z_0$ are just $(buv)^{\pm1}$, and
(b) $(b^{\mu'}k')(buv)^\epsilon(b^\mu k)^{-1}$ is not in~$\Pi$
for $\epsilon=0,1,2$, when $(\mu',k')$ and~$(\mu,k)$ in the above list
of six pairs are distinct.
\end{proof}

Finally, we show that $\Pi$ is a congruence subgroup of~$\Gamma$.

The prime 3 ramifies in~$\Q(\zeta)$ (as does~2), and $\F_9=\Z[\zeta]/r\Z[\zeta]$
is a field of order~9. Let $\rho:\Z[\zeta]\to\F_9$ be the natural map,
and write $i$ for~$\rho(\zeta)$. Then $i^2=-1$, and $\F_9=\F_3(i)$. Applying $\rho$ 
to matrix entries, we map $\Gamma$ to a group of matrices over~$\F_9$, modulo~$\langle i\rangle$. 
The image $\rho(g)$ of any $g\in M_{3\times3}(\Z[\zeta])$ unitary with respect 
to the~$F$ of~\eqref{eq:Fdefn} is unitary with respect to $\rho(F)$, and so if we conjugate
by $C=\rho(\gamma_0)$, where $\gamma_0$ is as in~\eqref{eq:variousmatrices}, then $\rho'(g)=C\rho(g)C^{-1}$
is unitary in the ``usual'' way. 

So $\rho'$ maps $\Gamma$ to the group $PU(3,\F_9)$ of unitary matrices with entries in~$\F_3(i)$, 
modulo scalars. This map is surjective. In fact, $\rho'(\Gamma_1)=PU(3,\F_9)$, where $\Gamma_1$
is the normal index~3 subgroup of~$\Gamma$ consisting of the $gZ\in\Gamma$ having a matrix
representative $g$ of determinant~1. One may check that $\Gamma_1=\langle v,bu^{-1},u^{-1}b\rangle$, 
and that $\langle\rho'(v),\rho'(bu^{-1}),\rho'(u^{-1}b)\rangle=PU(3,\F_9)$. 

The given generators $a_1$, $a_2$ and~$a_3$ of~$\Pi$ have determinants
$\zeta^3$, $\zeta^3$ and~$-1$, resp\-ect\-ive\-ly, and so $\Pi\subset\Gamma_1$. Now $-\zeta a_2$ and $-a_1a_2$ are mapped by $\rho'$ to the matrices
\begin{displaymath}
R=\begin{pmatrix}
 -i&-i-1&  i\\
  1& i-1& -1\\
i-1&   0&i-1
\end{pmatrix},
\quad\text{and}\quad
M=\begin{pmatrix}
   i&  -i&i+1\\
-i-1&   i& -i\\
   i&-i-1&  i
\end{pmatrix},
\end{displaymath}
respectively, which satisfy $R^7=I$, $M^3=I$ and $MRM^{-1}=R^2$. Moreover, $-a_3$ is mapped to~$R^{-1}$.
Hence $\Pi$ is mapped onto the subgroup $\langle R,M\rangle$ of $PU(3,\F_9)$, 
which has order~21. Now $|PU(3,\F_9)|=6048=288\times21$, and so the conditions on a $gZ\in\Gamma$ to
be in~$\Pi$ are that $gZ\in\Gamma_1$ and that $\rho'(g)\in\langle R,M\rangle$.
\end{section}

\begin{section}{Calculation of $r_0$.}
For any symmetric set $S\subset U(2,1)$, the following lemma simplifies the
description of the set~$\cF_S$ defined in~\eqref{eq:fsdefinition} in the case $X=B(\C^2)$.
\begin{lemma}\label{lem:fsdescription}If $g\in U(2,1)$ and $z=(z_1,z_2)\in B(\C^2)$, then
$d(0,z)\le d(0,g.z)$ if and only if 
$|g_{3,1}z_1+g_{3,2}z_2+g_{3,3}|\ge1$.
\end{lemma}
\begin{proof}Since $U(2,1)$ acts transitively on~$B(\C^2)$, we may write $z=h.0$ for some $h\in U(2,1)$.
So by~\eqref{eq:distg0to0}, $d(0,z)\le d(0,g.z)$ if and only if $|h_{3,3}|\le|(gh)_{3,3}|$.
Since $z_\nu=h_{\nu,3}/h_{3,3}$ for $\nu=1,2$, we have $(gh)_{3,3}=(g_{3,1}z_1+g_{3,2}z_2+g_{3,3})h_{3,3}$,
and the result follows.
\end{proof}
Now let $\Gamma$ and $S=K\cup KbK\cup Kbu^{-1}bK\subset\Gamma$ be as in Section~3. Write $r$ for~$+\sqrt{3}$.
For $1<\rho<(r+1)\sqrt{2}$, let $U_\rho$ denote the union of the 12 open discs in~$\C$ of radius~1
with centers $\rho\zeta^\lambda$, $\lambda=0,1,\ldots,11$. Let $B_\rho$ denote the
bounded component of~$\C\setminus U_\rho$. The conditions on~$\rho$ ensure that $B_\rho$ exists. See the diagram below.

Let $B_1$ and~$B_2$ denote $B_\rho$ for $\rho=(r+1)/\sqrt{2}$ and~$\rho=r+1$,
respectively. In the diagram, $\rho'$ and~$\rho''$ are the two solutions $t>0$ of
$|te^{i\pi/12}-\rho|=1$. When $\rho=(r+1)/\sqrt{2}$, we have $\rho'=1$ and $\rho''=r+1$.
When $\rho=r+1$, we have $\rho'=(r+1)/\sqrt{2}$ and
$\rho''=r(r+1)/\sqrt{2}$.

Write $\kappa$ for the square root of~$r-1$.

\begin{lemma}\label{lem:fsviacircles}Let $(w_1,w_2)\in\C^2$. Then $(w_1,w_2)\in\cF_S$ if and only if 
\begin{itemize}
\item[(i)] $u_1w_1+u_2w_2\in B_1$ for each of the pairs $(u_1,u_2)=(\sqrt{r+1},0)$, $(0,\sqrt{r+1})$
and $(\kappa^{-1}e^{-i\pi/12},\kappa^{-1}\zeta^{3\nu}e^{-i\pi/12})$ for $\nu=0,1,2,3$, and 
\item[(ii)] $u_1w_1+u_2w_2\in B_2$ for each of the pairs $(u_1,u_2)=(\kappa^{-1},\kappa^{-1}(\zeta+1)\zeta^{1+3\nu})$ 
and $(\kappa^{-1}(\zeta+1)\zeta^{1+3\nu},\kappa^{-1})$ for $\nu=0,1,2,3$,
\end{itemize}
in which case, $|w_1|,|w_2|\le1/\sqrt{r+1}$.
\end{lemma}

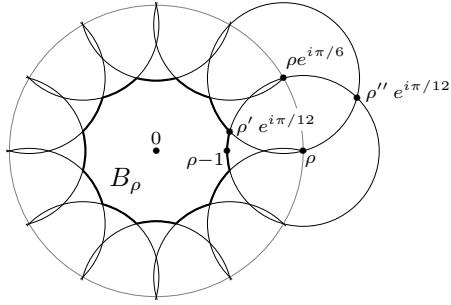
\begin{center}
\begin{pspicture}(-4,-2)(6,2)
\psset{xunit=1,yunit=1}%
\psarc[linewidth=0.1pt,linecolor=gray](0,0){1.93185165}{17}{368}
\psarc[linewidth=0.2pt](1.93185165,0){1}{0}{360}
\psarc[linewidth=0.2pt](1.67303260,0.965925826){1}{0}{360}
\psarc[linewidth=0.2pt](0.965925826,1.67303260){1}{162.5}{317.5}
\psarc[linewidth=0.2pt](0,1.93185165){1}{192.5}{347.5}
\psarc[linewidth=0.2pt](-0.965925826,1.67303260){1}{222.5}{377.5}
\psarc[linewidth=0.2pt](-1.67303260,0.965925826){1}{252.5}{407.5}
\psarc[linewidth=0.2pt](-1.93185165,0){1}{282.5}{437.5}
\psarc[linewidth=0.2pt](-1.67303260,-0.965925826){1}{312.5}{467.5}
\psarc[linewidth=0.2pt](-0.965925826,-1.67303260){1}{342.5}{497.5}
\psarc[linewidth=0.2pt](0,-1.93185165){1}{372.5}{527.5}
\psarc[linewidth=0.2pt](0.965925826,-1.67303260){1}{402.5}{557.5}
\psarc[linewidth=0.2pt](1.67303260,-0.965925826){1}{432.5}{587.5}
\psarc[linecolor=black](1.93185165,0){1}{165}{195}
\psarc[linecolor=black](1.67303260,0.965925826){1}{195}{225}
\psarc[linecolor=black](0.965925826,1.67303260){1}{225}{255}
\psarc[linecolor=black](0,1.93185165){1}{255}{285}
\psarc[linecolor=black](-0.965925826,1.67303260){1}{285}{315}
\psarc[linecolor=black](-1.67303260,0.965925826){1}{315}{345}
\psarc[linecolor=black](-1.93185165,0){1}{345}{375}
\psarc[linecolor=black](-1.67303260,-0.965925826){1}{375}{405}
\psarc[linecolor=black](-0.965925826,-1.67303260){1}{405}{435}
\psarc[linecolor=black](0,-1.93185165){1}{435}{465}
\psarc[linecolor=black](0.965925826,-1.67303260){1}{465}{495}
\psarc[linecolor=black](1.67303260,-0.965925826){1}{495}{525}
\rput(0,0){$\scriptscriptstyle\bullet$}
\rput(1.93185165,0.0){$\scriptscriptstyle\bullet$}
\rput(0.965925826,0.258819045){$\scriptscriptstyle\bullet$}
\rput(2.6389584337,0.7071067811){$\scriptscriptstyle\bullet$}
\rput(0.93185165,0){$\scriptscriptstyle\bullet$}
\rput(1.67303260,0.965925826){$\scriptscriptstyle\bullet$}
\rput[b](0,0.1){$\scriptstyle0$}
\rput[tl](1.95,-0.05){$\scriptstyle\rho$}
\rput[l](1.05,0.35){$\scriptstyle\rho'\,e^{i\pi/12}$}
\rput[l](2.75,0.8){$\scriptstyle\rho''\,e^{i\pi/12}$}
\rput[tr](0.9,-0.025){$\scriptstyle\rho-1$}
\rput[bl](1.7,1.1){$\scriptstyle\rho e^{i\pi/6}$}
\rput(-0.4,-0.4){$B_\rho$}
\end{pspicture}
\end{center}
\begin{proof}Given $w=(w_1,w_2)\in B(\C^2)$, to verify that $w\in\cF_S$, we must show that $d(0,w)\le d(0,(bk).w)$ and 
that $d(0,w)\le d(0,(bu^{-1}bk).w)$ for all $k\in K$. Since $b$ commutes with~$v$, and $bu^{-1}b$ 
commutes with~$u$, we must check $288/4+288/3=168$ conditions. 

Let $\gamma_0$ be as in~\eqref{eq:variousmatrices}, and let $D$, as before, be the diagonal 
matrix with diagonal entries 1, 1 and~$\kappa$. The $g\in U(2,1)$ to which we apply Lemma~\ref{lem:fsdescription} 
are the matrices $(bk)\,{\tilde{}}=D\gamma_0bk\gamma_0^{-1}D^{-1}$ and $(bu^{-1}bk)\,{\tilde{}}=D\gamma_0bu^{-1}bk\gamma_0^{-1}D^{-1}$, where $k\in K$. Now
\begin{displaymath}
(bk)\,{\tilde{}}_{3i}=\kappa(\gamma_0bk\gamma_0^{-1})_{3i}\quad \text{for}\quad i=1,2,\quad\text{and}\quad
(bk)\,{\tilde{}}_{33}=(\gamma_0bk\gamma_0^{-1})_{33},
\end{displaymath}
and similarly with $b$ replaced by $bu^{-1}b$. Note also that for $\lambda\in\Z$,
\begin{displaymath}
(\gamma_0bkj^\lambda \gamma_0^{-1})_{3i}=(\gamma_0bk\gamma_0^{-1})_{3i}\zeta^\lambda\ \text{for}\ i=1,2,\ \text{and}\ (\gamma_0bkj^\lambda \gamma_0^{-1})_{33}=(\gamma_0bk\gamma_0^{-1})_{33},
\end{displaymath}
and similarly with $b$ replaced by $bu^{-1}b$. So the conditions for $w=(w_1,w_2)\in\cF_S$ to hold have the form
\begin{displaymath}
|\kappa(g_{31}w_1+g_{32}w_2)\zeta^\lambda+g_{33}|\ge1\quad\text{for}\ \lambda=0,\ldots,11,
\end{displaymath}
for 6~matrices $g$ of the form $\gamma_0bk\gamma_0^{-1}$, and 8~matrices~$g$ of the form $\gamma_0bu^{-1}bk\gamma_0^{-1}$.
If $g=\gamma_0bk\gamma_0^{-1}$, then $g_{33}=(\zeta+1)/\zeta$, and if $g=\gamma_0bu^{-1}bk\gamma_0^{-1}$, then $g_{33}=r+1$.

By taking $k=uv^2u^{-1}j$, $j^{-2}$, and $uv^{2-\nu}j^{3(\nu-1)}$, for~$\nu=0,1,2,3$, respectively, we get from $g=\gamma_0bk\gamma_0^{-1}$ the
triples $(g_{31},g_{32},g_{33})$ equal to $((\zeta+1)\zeta^{-1},0,(\zeta+1)\zeta^{-1})$,
$(0,(\zeta+1)\zeta^{-1},(\zeta+1)\zeta^{-1})$
and $((r+1)\zeta^{-1}/2,(r+1)\zeta^{-1}\zeta^{3\nu}/2,(\zeta+1)\zeta^{-1})$.
Using $\zeta+1=\frac{r+1}{\sqrt{2}}e^{i\pi/12}$ and $\kappa(r+1)/2=\kappa^{-1}$, and replacing $\lambda$ by $6-\lambda$, we see that
the conditions coming from the six $g$ of the form $\gamma_0bk\gamma_0^{-1}$ are just the conditions $u_1w_1+u_2w_2\not\in U_\rho$
for $\rho=(r+1)/\sqrt{2}$ for the six~$(u_1,u_2)$ listed in~(i). Taking the case $(u_1,u_2)=(\sqrt{r+1},0)$, 
if $u_1w_1+u_2w_2=\sqrt{r+1}\,w_1$ is in the unbounded  component of~$\C\setminus U_\rho$, then $\sqrt{r+1}\,|w_1|\ge r+1$
(since $\rho''$ equals~$r+1$ in this case), and so $|w_1|\ge\sqrt{r+1}>1$,
which is impossible for $(w_1,w_2)\in B(\C^2)$. So $\sqrt{r+1}\,w_1$ is in the bounded component~$B_\rho$,
and so $|w_1|\le1/\sqrt{r+1}$ (since $\rho'$ equals~1 in this case). Similarly,
$|w_2|\le1/\sqrt{r+1}$ for all $(w_1,w_2)\in\cF_S$.

By taking $k=v^{1-\nu}j^{3\nu-1}$, and $k=vu^{-1}v^{2+\nu}j^9$, for~$\nu=0,1,2,3$, respectively, we get from $g=\gamma_0bu^{-1}bk\gamma_0^{-1}$ the
triples $(g_{31},g_{32},g_{33})$ equal to
$((r+1)/2,(r+1)(\zeta+1)\zeta^{1+3\nu}/2,r+1)$ and $((r+1)(\zeta+1)\zeta^{1+3\nu}/2,(r+1)/2,r+1)$,
$\nu=0,1,2,3$. Replacing $\lambda$ by $6-\lambda$, we see that
the conditions coming from the eight $g$ of the form $\gamma_0bu^{-1}bk\gamma_0^{-1}$ are just the conditions $u_1w_1+u_2w_2\not\in U_\rho$
for $\rho=r+1$ for the eight~$(u_1,u_2)$ listed in~(ii). Using $|w_1|,|w_2|\le1/\sqrt{r+1}$ for $(w_1,w_2)\in\cF_S$,
we see that $u_1w_1+u_2w_2$ is in the bounded component~$B_\rho$ of~$\C\setminus U_\rho$ in each case.
\end{proof}

So calculation of~$r_0$ in this case is equivalent to calculation
of the maximum value~$\rho_0$, say, of $|w|$ on the set of $w=(w_1,w_2)\in\C^2$
satisfying the conditions~(i) and~(ii) in Lemma~\ref{lem:fsviacircles}, and $r_0=\frac{1}{2}\log\Bigl(\frac{1+\rho_0}{1-\rho_0}\Bigr)$.
As we have seen, $|w_1|,|w_2|\le1/\sqrt{r+1}$ for $(w_1,w_2)\in\cF_S$. So $\cF_S$ is compact,
and $\rho_0\le\sqrt{r-1}$.

We can now show that the value of $r_0$ is $\frac{1}{2}d(\gamma_3.0,0)=\frac{1}{2}\cosh^{-1}(r+1)$, where $\gamma_3=bu^{-1}b$. 
We first prove that this is a lower bound for~$r_0$.

\begin{lemma}\label{lem:rzeroestimate}For the above~$S$, we have 
$r_0\ge\frac{1}{2}d(\gamma_3.0,0)=\frac{1}{2}\cosh^{-1}(r+1)$. That is, $\rho_0\ge(r-1)\sqrt{r/2}$.
\end{lemma}
\begin{proof}Consider the geodesic $[0,\gamma_3.0]$ from~0 to~$\gamma_3.0$, and let $m$ be the 
point on~$[0,\gamma_3.0]$ equidistant between~0 and~$\gamma_3.0$. Let us show that $m\in\cF_S$.
If $m\not\in\cF_S$, there is a $g\in S$ so that $d(g.0,m)<d(0,m)$.
Now $g.0\ne0$, so that $g\not\in K$. Also,
\begin{displaymath}
d(g.0,0)\le d(g.0,m)+d(m,0)<2d(m,0)=d(\gamma_3.0,0),
\end{displaymath}
and so $g\not\in K\gamma_3K$. So $g$ must be in~$K\gamma_2K=KbK$. Since $m\in[0,\gamma_3.0]$,
\begin{displaymath}
d(g.0,\gamma_3.0)\le d(g.0,m)+d(m,\gamma_3.0)<d(0,m)+d(m,\gamma_3.0)=d(0,\gamma_3.0),
\end{displaymath}
so that $g^{-1}\gamma_3\in K\cup KbK$. Now $g^{-1}\gamma_3\not\in K$, since otherwise
$g.0=\gamma_3.0$, so that $m$ is closer to~$\gamma_3.0$ than to~0. Since $KbK$ is symmetric, we have $\gamma_3^{-1}g\in KbK$. Thus
$g$ must be in $\cG=KbK\cap\,\gamma_3KbK$.  One may verify that $\cG=(\langle u\rangle bK)\cup(\langle u\rangle b^{-1}K)$.
Since $u$ is in~$K$ and commutes with~$\gamma_3$, it fixes~$[0,\gamma_3.0]$, and so $d(g.0,m)$ is constant on both double
cosets $\langle u\rangle bK$ and $\langle u\rangle b^{-1}K$. Note that $u\gamma_3^{-1}=\gamma_3u^{-1}$ is an element in~$\Gamma$ of order~2 which
interchanges~0 and~$\gamma_3.0$, and so fixes~$m$. The map $f:g\mapsto u\gamma_3^{-1}g$ is an involution
of $\cG$, and $d(f(g).0,m)=d(g.0,m)$. Also, $f(b)=ub^{-1}u$, so that $f$ interchanges the two double
cosets, and so $d(g.0,m)$ is constant on~$\cG$. So to show the result, we need only check that $d(b.0,m)<d(0,m)$ 
does not hold. 
Now $\gamma_3.0=(z_1,z_2)$ for $z_1=-(r-1)\zeta^2/(2\sqrt{r-1})$ and $z_2=(i+1)\zeta^2/(2\sqrt{r-1})$.
Write $m_t=(tz_1,tz_2)$ for $0\le t\le1$. Then $m=m_t$ for 
$t=(1-\sqrt{1-|\gamma_3.0|^2}\,)/|\gamma_3.0|^2=(1-\sqrt{1-r/2}\,)/(r/2)=r-1$. Some routine 
calculations show that $b.m=m$, and so $d(b.0,m)=d(0,m)$. So $m\in\cF_S$, and $r_0\ge d(0,m)=\frac{1}{2}d(0,\gamma_3.0)$.
\end{proof}
\begin{proposition}\label{prop:rho0bound}If $(w_1,w_2)\in\cF_S$, then $|w_1|^2+|w_2|^2\le2r-3=\frac{r}{2}(r-1)^2$.
\end{proposition}
\begin{proof}Let $\cF_{S^*}=\{z\in B(\C^2):d(0,z)\le d(g.0,z)\ \text{for all}\ g\in S^*\}$
for $S^*=K\cup KbK$.
Since $S^*\subset S$, we have $\cF_S\subset\cF_{S^*}$. We shall in fact show
that $|w_1|^2+|w_2|^2\le2r-3$ for $(w_1,w_2)\in\cF_{S^*}$.

Suppose that 
$w=(w_1,w_2)\in\cF_{S^*}$ and $|w_1|^2+|w_2|^2$ is maximal. Then $d(0,w)=d(0,g.w)$ for 
some $g\in KbK$. Replacing $w$ by $k.w$ for some $k\in K$, if necessary, we may suppose that
$g=buv^2u^{-1}j^7$. Let $\tilde g=D\gamma_0g\gamma_0^{-1}D^{-1}\in U(2,1)$ for this~$g$. Since $\tilde g_{31}=-\kappa(\zeta+1)\zeta^{-1}$, $\tilde g_{32}=0$
and $\tilde g_{33}=(\zeta+1)\zeta^{-1}$, Lemma~\ref{lem:fsdescription} shows that $1=|\tilde g_{31}w_1+\tilde g_{32}w_2+\tilde g_{33}|=
|\zeta+1|\,|\sqrt{r-1}\,w_1-1|$. Since $\sqrt{r+1}\,w_1\in B_1$, this means that
$\sqrt{r+1}\,w_1$ is on the rightmost arc of the boundary of~$B_1$. That is,
\begin{displaymath}
w_1=\frac{1}{\sqrt{r-1}}-\frac{1}{\sqrt{r+1}}e^{-i\theta}\quad\text{for some}\ \theta\in[-\pi/12,\pi/12].
\end{displaymath}
Fixing~$w_1$, we see from Lemma~\ref{lem:fsviacircles}(i) that $w_2$ must lie on or outside various circles, including, for $\nu=0,1,2,3$ and $\epsilon=\pm$, 
the circles~$C_{\nu,\epsilon}$ of radius~$\sqrt{r-1}$ and center
\begin{displaymath}
\alpha_{\nu,\epsilon}=-i^\nu\bigl(w_1-\sqrt{r+1}\,e^{\epsilon i\pi/12}\bigr)=i^\nu\bigl(e^{\epsilon i\pi/4}+e^{-i\theta}\bigr)/\sqrt{r+1}\,.
\end{displaymath}
In the following diagram, we have taken $\theta=\pi/24$:
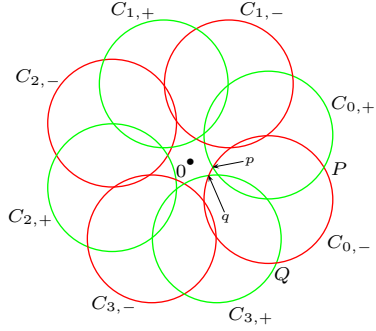
\begin{center}
\begin{pspicture}(-3,-2.25)(3,2.25)
\psset{xunit=1,yunit=1}%
\rput(0,0){$\scriptscriptstyle\bullet$}
\rput[tr](-0.05,-0.05){$\scriptstyle0$}
\pscircle[linewidth=0.5pt,linecolor=red](-1.02762431057,0.506768228434){0.855599677167}
\pscircle[linewidth=0.5pt,linecolor=green](-1.02762431057,-0.348831448733){0.855599677167}
\pscircle[linewidth=0.5pt,linecolor=red](-0.506768228434,-1.02762431057){0.855599677167}
\pscircle[linewidth=0.5pt,linecolor=green](0.348831448733,-1.02762431057){0.855599677167}
\pscircle[linewidth=0.5pt,linecolor=red](1.02762431057,-0.506768228434){0.855599677167}
\pscircle[linewidth=0.5pt,linecolor=green](1.02762431057,0.348831448733){0.855599677167}
\pscircle[linewidth=0.5pt,linecolor=red](0.506768228434,1.02762431057){0.855599677167}
\pscircle[linewidth=0.5pt,linecolor=green](-0.348831448733,1.02762431057){0.855599677167}
\rput(2.1,-1.1){$\scriptstyle{C_{0,-}}$}
\rput(2.15,0.75){$\scriptstyle{C_{0,+}}$}
\rput(0.95,2){$\scriptstyle{C_{1,-}}$}
\rput(-0.75,2){$\scriptstyle{C_{1,+}}$}
\rput(-2,1.1){$\scriptstyle{C_{2,-}}$}
\rput(-2.1,-0.75){$\scriptstyle{C_{2,+}}$}
\rput(-1,-1.9){$\scriptstyle{C_{3,-}}$}
\rput(0.8,-2.05){$\scriptstyle{C_{3,+}}$}
\rput[l](0.72,0){$\scriptscriptstyle{p}$}
\rput[l](1.85,-0.1){$\scriptstyle{P}$}
\psline[linewidth=0.2pt]{->}(0.7,0)(0.28665,-0.078968)
\rput[tl](0.4,-0.7){$\scriptscriptstyle{q}$}
\rput[tl](1.1,-1.4){$\scriptstyle{Q}$}
\psline[linewidth=0.2pt]{->}(0.45,-0.675)(0.237153,-0.179344)
\end{pspicture}
\end{center}
Let $U_\epsilon(\theta)$ denote the union of the four open discs bounded by the circles $C_{\nu,\epsilon}$, $\nu=0,1,2,3$.
Using $0<\cos(\theta+\epsilon\pi/4)<1$, we see that $\C\setminus U_\epsilon(\theta)$ has two components, the 
bounded one containing~0. So the complement of $U(\theta)=U_+(\theta)\cup U_-(\theta)$ has a
bounded component containing~0. The set~$U(\theta)$ is obviously invariant under the rotations $z\mapsto i^\lambda z$, $\lambda=0,1,2,3$.
It is also invariant  under the reflection $R_{\nu,\epsilon}$ in the line through~0 
and~$\alpha_{\nu,\epsilon}$, for each~$\nu$ and~$\epsilon$. For
if $0\ne\alpha\in\C$, the reflection in the line through~0 and~$\alpha$ is the map
$R_\alpha:z\mapsto \alpha\bar z/\bar\alpha$. 
It is then easy to check that
\begin{displaymath}
R_{\nu,\epsilon}(\alpha_{\nu',\epsilon'})=\alpha_{\nu'',\epsilon'}\quad\text{for}\ \nu''=2\nu-\nu'+(\epsilon-\epsilon')/2\ \text{(mod 4).}
\end{displaymath}
Let $p$ and~$P$ be the points of intersection of $C_{0,+}$ and~$C_{0,-}$. Then
$R_{0,-}(C_{0,+})=C_{3,+}$ and $R_{0,-}(C_{0,-})=C_{0,-}$, so that $R_{0,-}(p)$ and~$R_{0,-}(P)$ are the
points $q$ and~$Q$ of intersection of~$C_{0,-}$ and~$C_{3,+}$. In particular, $|q|=|p|$ and $|Q|=|q|$.

It is now clear that $|z|\le|p|$ for all $z$ in the
bounded component of $\C\setminus U(\theta)$, and that $|z|\ge|P|$ for all $z$ in the unbounded
component. 

We now evaluate~$|p|$ and~$|P|$. Let $z\in C_{0,+}\cap C_{0,-}$. Then
\begin{displaymath}
\bigl|z+\bigl(w_1-\sqrt{r+1}\,e^{i\pi/12}\bigr)\bigr|=\sqrt{r-1}=
\bigl|z+\bigl(w_1-\sqrt{r+1}\,e^{-i\pi/12}\bigr)\bigr|.
\end{displaymath}
For $\alpha,\beta\in\C$ with $\beta\not\in\R$, $|\alpha-\beta|=|\alpha-\bar\beta|$ if and only if $\alpha\in\R$. 
So $z+w_1$ must be real, and writing $z+w_1=t\sqrt{r+1}$, with $t\in\R$, we have
\begin{displaymath}
\sqrt{r-1}=\bigl|t\sqrt{r+1}-\sqrt{r+1}\,e^{i\pi/12}\bigr|, 
\end{displaymath}
so that 
\begin{displaymath}
t^2-\frac{r+1}{\sqrt{2}}t+r-1=0.
\end{displaymath}
The solutions of this are $t=(r-1)/\sqrt{2}$ and $t=\sqrt{2}$.
Taking the smaller of these,
\begin{displaymath}
p+w_1=\frac{r-1}{\sqrt{2}}\sqrt{r+1}=\sqrt{r-1}.
\end{displaymath}
So $|p|=|\sqrt{r-1}-w_1|$. Taking instead $t=\sqrt{2}$, we see that 
$|P|=|\sqrt{2}\sqrt{r+1}-w_1|$. So $|P|\ge\sqrt{2}\sqrt{r+1}-1/\sqrt{r+1}>1/\sqrt{r+1}$,
and therefore $|z|>1/\sqrt{r+1}$ for all $z$ in the unbounded component of~$\C\setminus U(\theta)$.

So $(w_1,w_2)\in\cF_{S^*}$ implies that $w_2$ is in the bounded component of~$\C\setminus U(\theta)$, and
therefore $|w_2|\le|p|=|\sqrt{r-1}-w_1|$. Thus
\begin{displaymath}
\begin{aligned}
|w_1|^2+|w_2|^2&\le|w_1|^2+|\sqrt{r-1}-w_1|^2\\
&=r-1+2|w_1|^2-2\sqrt{r-1}\,\mathrm{Re}(w_1)\\
&=r-1+2\Bigl(\frac{1}{r-1}+\frac{1}{r+1}-\sqrt{2}\cos\theta\Bigr)\\
&\qquad-2\sqrt{r-1}\Bigl(\frac{1}{\sqrt{r-1}}-\frac{1}{\sqrt{r+1}}\cos\theta\Bigr)\\
&=3(r-1)-r(r-1)\sqrt{2}\,\cos\theta\\
&\le3(r-1)-r(r-1)\sqrt{2}\,\frac{r+1}{\sqrt{8}}\\
&=2r-3.\\
\end{aligned}
\end{displaymath}
\end{proof}

We conclude by giving a direct proof that the set~$S$ generates~$\Gamma$. The following
lemma uses a modification of an argument shown to us by Gopal Prasad.
\begin{lemma}\label{lem:mindistance}If $g\in\Gamma\setminus K$, then $d(g.0,0)\ge d(b.0,0)=\cosh^{-1}(\sqrt{r+2}\,)$.
\end{lemma}
\begin{proof}
By considering the $(3,3)$-entry of $g^*Fg-F=0$, we see that
\begin{equation}\label{eq:col3cond}
|g_{13}|^2+|g_{13}-(r-1)g_{23}|^2=(r-1)\bigl(|g_{33}|^2-1\bigr).
\end{equation}
Write $\alpha=g_{13}$, $\beta=g_{13}-(r-1)g_{23}$ and $\gamma=g_{33}$. 
By hypothesis, $g.0\ne0$, and so $g_{13}\ne0$ or $g_{23}\ne0$. Hence
\begin{equation}\label{eq:alphabetagammaconds}
\alpha,\beta,\gamma\in\Z[\zeta],\quad
|\alpha|^2+|\beta|^2=(r-1)\bigl(|\gamma|^2-1\bigr),\quad\text{and}\
\alpha,\beta\ \text{are not both}\ 0.
\end{equation}
We claim that under conditions~\eqref{eq:alphabetagammaconds}, $|\gamma|^2\ge r+2$ must hold.
Writing $\alpha=a_0+a_1\zeta+a_2\zeta^2+a_3\zeta^3$, we have 
$|\alpha|^2=P(\alpha)+rQ(\alpha)$, where $P(\alpha)=a_0^2+a_0a_2+a_2^2+a_1^2+a_1a_3+a_3^2$
and $Q(\alpha)=a_0a_1+a_1a_2+a_2a_3$. Our hypothesis is that 
\begin{equation}\label{eq:PQconds}
P(\alpha)+P(\beta)+P(\gamma)=3Q(\gamma)+1\quad\text{and}\quad P(\gamma)=Q(\alpha)+Q(\beta)+Q(\gamma)+1,
\end{equation}
and we want to show that $P(\gamma)+rQ(\gamma)\ge 2+r$. Now $P$ is a positive definite form,
and $\gamma\ne0$ and either $\alpha\ne0$ or $\beta\ne0$. So the left hand side of the first
equation in~\eqref{eq:PQconds} is at least~2, so that $Q(\gamma)>0$. Since $Q(\gamma)\in\Z$,
we have $Q(\gamma)\ge1$. So all we have to do is show that $P(\gamma)\ge2$. Now $Q(\gamma)\ne0$
implies that $a_0a_1$, $a_1a_2$ or $a_2a_3\ne0$. If $a_0a_1\ne0$, then $P(\alpha)=(a_2+a_0/2)^2+3a_0^2/4
+(a_3+a_1/2)^2+3a_1^2/4\ge3(a_0^2+a_1^2)/4\ge3/2>1$, so that $P(\alpha)\ge2$. The other two cases
are similar.
\end{proof}
The following result implies that $\Gamma$ is generated by~$K$ and~$b$.
\begin{lemma}If $g\in\Gamma\setminus K$, then there exists $k\in K$ so that $d(bkg.0,0)<d(g.0,0)$.
\end{lemma}
\begin{proof}If there is no such~$k\in K$, then $d(s.(g.0),0)\ge d(g.0,0)$ for all $s\in S^*$,
and so $g.0\in\cF_{S^*}$, in the notation of the proof of Proposition~\ref{prop:rho0bound}, and
so $d(g.0,0)\le r_0$. But by Lemma~\ref{lem:mindistance},
$d(g.0,0)\ge d(b.0,0)=\cosh^{-1}(\sqrt{r+2}\,)$, which is greater than $\frac{1}{2}\cosh^{-1}(r+1)=r_0$.
\end{proof}
\end{section}

\bibliographystyle{amsplain}
\bibliography{short}

\end{document}